%% file: paper.tex
\RequirePackage{rotating} 
\RequirePackage{fix-cm}
\documentclass[smallextended]{svjour3}       
\smartqed  
\usepackage{article_packages_mathopt}
\usepackage{standard_macros_mathopt}
\usepackage{listings}
\usepackage{stackengine}
\usepackage[linesnumbered, ruled]{algorithm2e}
\usepackage{multicol}
\usepackage{rotating}
\usepackage{booktabs}
\usepackage{siunitx}
\newcommand{\newcontent}[1]{#1}

\newcommand{\newcontentB}[1]{#1}

\SetKwProg{Fn}{Function}{}{}
\makeatletter
\newcommand{\nosemic}{\renewcommand{\@endalgocfline}{\relax}}
\newcommand{\dosemic}{\renewcommand{\@endalgocfline}{\algocf@endline}}

\makeatother

\usepackage[hang,flushmargin]{footmisc}
\usepackage{lipsum}
\makeatletter
\newcommand{\algorithmfootnote}[2][\footnotesize]{%
  \let\old@algocf@finish\@algocf@finish
  \def\@algocf@finish{\old@algocf@finish
    \leavevmode\rlap{\begin{minipage}{\linewidth}
    #1#2
    \end{minipage}}%
  }%
}
\makeatother

\renewcommand{\eqdef}{:=}

\journalname{Mathematical Programming}

\begin{document}
\title{An active-set algorithm for norm constrained quadratic problems}
\author{Nikitas {Rontsis} \and
    Paul J. Goulart \and \\
    Yuji Nakatsukasa
}

\institute{Nikitas {Rontsis} \and Paul J. Goulart \at
Department of Engineering Science,
University of Oxford,
Oxford, OX1 3PN, UK \\
\email{nrontsis@gmail.com, paul.goulart@eng.ox.ac.uk}
\and
Yuji Nakatsukasa \at
Mathematical Institute,  University of Oxford, Oxford, OX2 6GG, UK, and National Institute of Informatics, Japan. \\
\email{nakatsukasa@maths.ox.ac.uk}
}
\maketitle

\begin{abstract}
We present an algorithm for the minimization of a nonconvex quadratic function subject to linear inequality constraints and a two-sided bound on the 2-norm of its solution. The algorithm minimizes the objective using an active-set method by solving a series of Trust-Region Subproblems (TRS). Underpinning the efficiency of this approach is that the global solution of the TRS has been widely studied in the literature, resulting in remarkably efficient algorithms and software. We extend these results by proving that nonglobal minimizers of the TRS, or a certificate of their absence, can also be calculated efficiently by computing the two rightmost eigenpairs of an eigenproblem. We demonstrate the usefulness and scalability of the algorithm in a series of experiments that often outperform state-of-the-art approaches; these include calculation of high-quality search directions arising in Sequential Quadratic Programming on problems of the \texttt{CUTEst} collection, and Sparse Principal Component Analysis on a large text corpus problem (70 million nonzeros) that can help organize documents in a user interpretable way.
\end{abstract}

\input{sections/intro}
\input{sections/trs}
\input{sections/algorithm}

\input{sections/results_first}
\input{sections/results_cutest}
\input{sections/results_pca}
\input{sections/tables_pca}

\section{Conclusion}
This paper introduced an active-set algorithm for the solution of quadratic functions subject to linear constraints and a single norm constraint. The suggested algorithm is based on repeated solutions of the TRS, for which we derived novel theoretical results regarding its local-nonglobal minimizer. The usefulness of the suggested algorithm was demonstrated in a range of real world experiments.

\paragraph{Acknowledgements:} We are grateful to the detailed feedbacks from an anonymous reviewer that significantly improved the quality of this document. Furthermore, we thank Nick Trefethen for providing valuable comments. This work was supported by the EPSRC AIMS CDT grant EP/L015987/1 and Schlumberger. We would also like to acknowledge the support of the NII International Internship Program, which funded the first author's visit to NII, during which this collaboration was launched. We thank Ron Estrin for valuable discussions. We used the University of Oxford Advanced Research Computing (ARC) facility in carrying out this work. http://dx.doi.org/10.5281/zenodo.22558.
\bibliographystyle{spmpsci}
\bibliography{../../Bibliography/bibliography}
\appendix
\section*{Appendix} \label{appendix}
\input{sections/appendix_pca}
\input{sections/tables_cutest}

\end{document}

%% file: sections/intro.tex
\section{Introduction}
Optimization with spherical $\norm{x}_2 = 1$ or norm constraints $r_{\min} \leq \norm{x}_2 \leq r_{\max}$ is used in a number of scientific fields. Many optimization problems, such as eigenvalue problems \cite{Saad2011}, dimensionality reduction \cite{Jolliffe2011} and compressed sensing \cite{Boufounos2008}, are naturally posed on or in the unit sphere, while norm constraints are often useful for regularization, e.g.\ in general nonlinear \cite{Conn2000} or robust optimization \cite{Jeyakumar2014}. In this paper we consider the solution of norm constrained problems in the form
\begin{equation}\label{eqn:main_problem} \tag{P}
	\begin{array}{ll}
		\text{minimize}   & f(x) \eqdef \frac{1}{2}x^T P x + q^T x \\
		\text{subject to} & r_{\min} \leq \norm{x}_2 \leq r_{\max} \\
		                  & Ax \leq b
	\end{array}
\end{equation}
where $x \in \mathbb{R}^n$ is the decision variable, $P \in \mathbb{S}^{n}$, i.e.\ an $n \times n$ symmetric matrix, $A = [a_1^T \dots a_m^T] \in \mathbb{R}^{m \times n}$, $q \in \mathbb{R}^{n}$, $b = [b_1 \dots b_m]^T \in \mathbb{R}^m$ and $r_{\min}, r_{\max}$ non-negative scalars.

It is generally difficult to solve problems of the form \eqref{eqn:main_problem} due to the nonconvexity of the lower bounding norm constraint and the potential indefiniteness of the matrix $P$, which renders many, but certainly not all \cite{Jeyakumar2014}, of these problems intractable. Even finding a feasible point for \eqref{eqn:main_problem} or testing if a first-order critical point of \eqref{eqn:main_problem} is a local minimizer can be intractable (see Proposition \ref{prop:np_feasibility} and \cite{Gould2005}). As a result, we restrict our attention to the search for first order critical points and we assume that a feasible point is given or can be computed.

A specific tractable variation of \eqref{eqn:main_problem} is the famous Trust-Region Subproblem\footnote{Note that the TRS is typically defined in the ball $\norm{x} \leq r$ rather than the sphere $\norm{x} = r$. We consider only the boundary solution as this fits better the needs of this paper.}:
\begin{equation} \label{eqn:trs_full} \tag{TRS}
	\begin{array}{ll}
		\text{minimize}   & \frac{1}{2}x^T  P x + q^T x \\
		\text{subject to} & \norm{x}_2 = r                   \\
											& Ax = b.
	\end{array}
\end{equation}
Perhaps surprisingly, the TRS can be solved to global optimality despite its nonconvexity, and global solution of the TRS has been widely studied in the literature. Specific approaches to its solution include exact Semidefinite Reformulations, Gradient Steps, Truncated Conjugate Gradient Steps, Truncated Lanczos Steps and Newton root-finding, with associated software packages for its solution. The reader can find details in \cite{Conn2000} and \cite[Chapter 4]{Nocedal2006}.

We suggest an active-set algorithm for the solution of \eqref{eqn:main_problem} that, as we proceed to describe, takes advantage of the existing efficient global solution methods for the TRS. Using such an approach, \eqref{eqn:main_problem} is optimized by solving a series of equality constrained subproblems. Each subproblem is a modification of \eqref{eqn:main_problem} where a subset of its inequalities (called the \emph{working-set}) is replaced by equalities while the rest are ignored. If a subproblem does not include a norm constraint then it is called an \emph{Equality-constrained Quadratic Problem} (EQP); otherwise it is a \emph{Trust Region Subproblem} (TRS). As described earlier, the global solution for the TRS, and also for the EQP \cite{Gould2001}, has been widely studied in the literature. Thus we can efficiently address the global solution of each subproblem. However, a complication exists for the TRS because, unlike EQP, it can exhibit \emph{local-nonglobal} minimizers i.e.\ local minimizers that are not globally optimal, which complicate the analysis. Indeed, if the optimal working-set includes a norm constraint and the global solutions of the respective TRS subproblem are infeasible for \eqref{eqn:main_problem}, then the solution of \eqref{eqn:main_problem} must be obtained from a local-nonglobal optimizer of the TRS subproblem.

Thus, an algorithm for computing local-nonglobal minimizers of the TRS is required for our active-set algorithm. Unfortunately, algorithms for detecting the presence of/computing local-nonglobal solutions of \eqref{eqn:trs_full} are significantly less mature than their global counterparts. A notable exception is the work of Martinez \cite{Martinez1994}, which proved that there exists at most one local-nonglobal minimizer and gave conditions for its existence. Moreover, \cite{Martinez1994} presented a root-finding algorithm for the computation of the Lagrange Multiplier associated with the local-nonglobal minimizer that entailed the computation of the two smallest eigenvalues of the Hessian and the solution of a series of indefinite linear systems.

More recent work is based on a result from \cite{Adachi2017}, which shows that each KKT point of the TRS, in the absence of linear equality constraints $Ax = b$ (i.e.\ with a norm constraint on $x$ only), can be extracted from an eigenpair of
\begin{equation*}
	M \eqdef
	\begin{bmatrix}
		-P & q q^T/r^2    \\
		I & -P
	\end{bmatrix}.
\end{equation*}
This result was used in \cite{Salahi2017} to calculate the local-nonglobal minimizer under the assumption that it exists. We extend these results by proving that both checking if the local-nonglobal minimizer exists and computing it, in the case where it exists, can be performed at the same time by simply calculating the two rightmost eigenpairs of $M$. Crucially, and similarly to the results of \cite{Adachi2017} for the global minimizer, this allows the use of efficient factorization-free methods such as the Arnoldi method for the detection and computation of the local-nonglobal minimizer. Furthermore, we show that this approach can efficiently handle equality constraints $Ax = b$ in the Trust Region Subproblem, without the need for variable reduction, by means of a projected Arnoldi method.

The paper is organised as follows. In Section \ref{sec:trs} we present an introduction to the TRS and present novel results for the detection/computation of its local-nonglobal minimizer and for the incorporation of linear equality constraints. In Section \ref{sec:active_set} we introduce an active-set algorithm for solving \eqref{eqn:main_problem} starting from the special case where $r_{\min} = r_{\max}$, and then generalizing to any $r_{\min}, r_{\max}$. Finally, in Section \ref{sec:experiments}, we conclude with a series of numerical results.

\paragraph{Notation used:} $\mathbb{S}^n$ denotes the set of symmetric $n\times n$ real matrices. Given a vector (or matrix) $x$, $x^T$ denotes its transpose and $x^H$ its conjugate transpose.

%% file: sections/trs.tex
\section{The Trust-Region Subproblem} \label{sec:trs}
This section concerns the computation of global as well as local-nonglobal minimizers for \eqref{eqn:trs_full}. Although this section only considers solutions on the boundary $\norm{x} = r$, the results can be trivially extended to the interior $\norm{x} \leq r$. Indeed, if an interior solution exists then the TRS is essentially convex ($P$ is positive semidefinite in the nullspace of $A$ \cite[\S1]{Gould1999}), thus no local-nonglobal solution exists and the interior solution is obtained by solving an equality constrained quadratic problem. The presence of a (necessarily global) interior solution can be detected by checking the sign of the Lagrange Multiplier of the global boundary solution(s) \cite[Lemma 2.2]{Martinez1994}.

For clarity of exposition, we will first assume that there are no linear equality constraints in \eqref{eqn:trs_full}, resulting in the following problem:
\begin{equation} \tag{T} \label{eqn:trs}
	\begin{array}{ll}
		\text{minimize}   & \frac{1}{2}x^T P x + q^T x \\
		\text{subject to} & \norm{x}_2 = r,                   \\
	\end{array}
\end{equation}
where $x \in \mathbb{R}^n$ is the decision variable, and $P \in \mathbb{S}^n$, $q \in \mathbb{R}^n, r \in \mathbb{R}_{+}$ are the problem data. For the rest of the section we will assume that $n > 1$, excluding the trivial one dimensional case where the feasible set consists of two points, and that $r > 0$. We will then show in \S\ref{subsec:trs_full} how to extend the results of this section to allow for the inclusion of linear equality constraints. Finally, in \S\ref{sec:trs_proof} we prove the main result of this section, which is presented in a separate subsection to facilitate the presentation.

In the sequel we will make frequent use of the eigendecomposition of $P \eqdef W \Lambda W^T$ defined by an orthonormal matrix $W \eqdef [w_1 \dots w_n]$ and $\Lambda \eqdef \mathrm{diag}(\lambda_1, \dots, \lambda_n)$ where $\lambda_1 \leq \dots \leq \lambda_n$.

Every KKT point of \eqref{eqn:trs} satisfies
\begin{equation} \label{eqn:trs_stationarity}
	P x + q + \mu x = 0 \Rightarrow x = -(P + \mu I)^{-1} q,
\end{equation}
where $\mu$ is a Lagrange multiplier associated with the norm constraint\footnote{Technically, $\mu$ is the Lagrange multiplier of the equivalent constraint $\frac{1}{2} \norm{x}_2^2 = \frac{1}{2} r^2$. We define \eqref{eqn:trs} with $\norm{x}_2 = r$ for simplicity and to match the notation of \cite{Adachi2017}.}. Equation \eqref{eqn:trs_stationarity} is well defined when $\mu \neq -\lambda_i, \; i = 1, \dots, n$. For the purposes of clarity of the subsequent introduction, which follows \cite[\S 4]{Nocedal2006}, we will assume that this holds. However, the conclusions of this section are independent of this assumption.

Using the feasibility of $x$, we have
\begin{equation} \label{eqn:stationarity_trs}
	\norm{x}_2^2 = r^2 \Leftrightarrow \norm{(P + \mu I)^{-1}q}_2^2 - r^2 = 0.
\end{equation}
Then, noting that
\begin{align*}
	\norm{(P + \mu I)^{-1}q}_2^2
	  & = \norm{(W(\Lambda + \mu I)W^T)^{-1} q}_2^2
	= \norm{W(\Lambda + \mu I)^{-1} W^T q}_2^2 \\
	  & = \norm{(\Lambda + \mu I)^{-1} W^T q}_2^2
	= \sum_{i=1}^n{\frac{(w_i^T q)^2}{(\lambda_i + \mu)^2}},
\end{align*}
we can express the rightmost condition in \eqref{eqn:stationarity_trs} as
\begin{equation} \label{eqn:secular}
	s(\mu) \eqdef \sum_{i=1}^{n}\frac{(w_i^T q)^2}{(\lambda_i + \mu)^2} - r^2 = 0.
\end{equation}
Determining the KKT points of \eqref{eqn:trs} is therefore equivalent to finding the roots of $s$, which is often referred to as the \emph{secular equation} \cite{Nocedal2006}. We depict $s$ for a particular choice of $P, q, r$ in Figure \ref{fig:secular}. As might be expected from the tight connection between polynomial root-finding problems and eigenproblems, solving $s(\mu) = 0$ is equivalent to the following eigenproblem \cite[Equation (22)]{Adachi2017}:
\begin{equation} \label{eqn:eigenvalue_problem}
	\underbrace{
	\begin{bmatrix}
		-P & q q^T/r^2    \\
		I & -P
	\end{bmatrix}}_{\eqdef M}
	\begin{bmatrix}
		z_1 \\
		z_2
	\end{bmatrix}
	= \mu
	\begin{bmatrix}
		z_1 \\
		z_2
	\end{bmatrix}.
\end{equation}
This elegant result was first noted more than 50 years ago \cite[Equation (2.21)]{Forsythe1965}, only to be later disregarded as inefficient by some of the same authors \cite{Gander1989}, and then rediscovered by \cite{Adachi2017} who highlighted its great applicability owing to the remarkable efficacy of modern eigensolvers \cite{Lehoucq1998}.
\begin{figure}
	\centering
	\includegraphics[width=.9\textwidth]{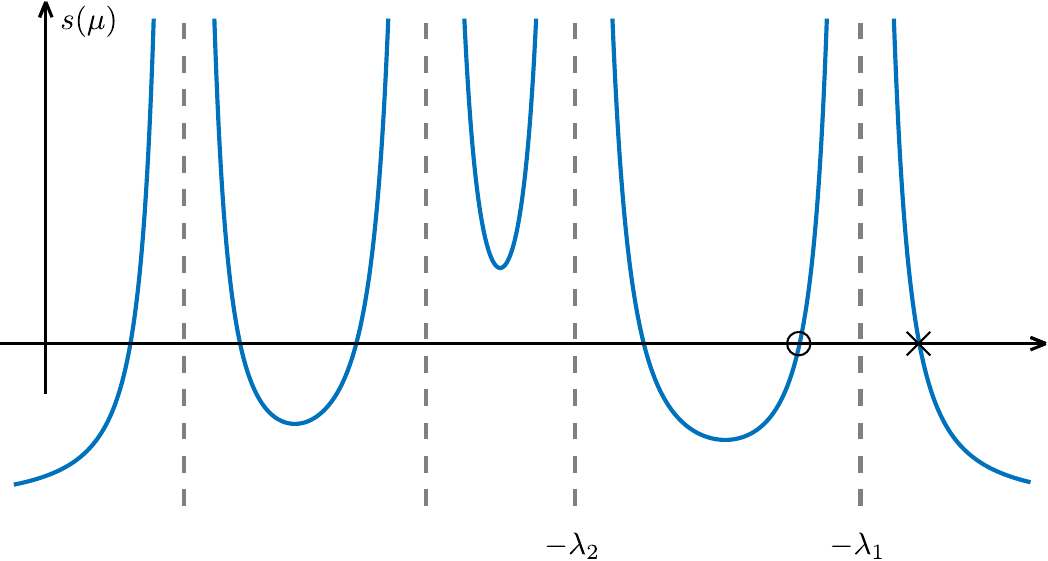}
	\caption{A typical example of the secular equation $s(\mu)$. The dashed vertical lines mark the locations of the eigenvalues of $-P$, i.e.\ $-\lambda_1, \dots, -\lambda_n$. The Lagrange multiplier of the global solution of \eqref{eqn:trs} is the rightmost root of $s(\mu)$ (marked by a cross). The TRS associated with this Figure also exhibits a local-nonglobal minimizer with a Lagrange multiplier that corresponds to the second-rightmost root of $s(\mu)$ (marked by a circle).}
	\label{fig:secular}
\end{figure}

The relation between the spectrum of $M$ and the Lagrange multipliers of \eqref{eqn:trs} is formally stated below:
\begin{lemma} \label{lem:multipliers}
	The spectrum of $M$ includes the Lagrange multiplier of every KKT point of \eqref{eqn:trs}.
\end{lemma}
\begin{proof}
	This is a consequence of \newcontent{Proposition \ref{prop:spectrum_new}} of subsection \ref{sec:trs_proof} and \cite[Theorem 4.1]{Forsythe1965}.
\qed\end{proof}
Note that the eigenvalues of $M$ are complex in general since $M$ is asymmetric. Furthermore, Lemma \ref{lem:multipliers} and all the subsequent results allow for potential ``degenerate'' KKT points with $\mu = -\lambda_i$.

We will focus on the KKT points that are minimizers of \eqref{eqn:trs} rather than maximizers or saddle points. The characterization of global minimizers is widely known in the literature. The following theorem shows that the Lagrange multiplier $\mu^g$ of the global minimizer corresponds to the rightmost eigenvalue of $M$ in $\mathbb{C}$, which is always real:
\begin{theorem} \label{thm:trs_global_min}
	A KKT point of \eqref{eqn:trs} with Lagrange multiplier $\mu^g$ is a global minimizer if and only if $\mu^g$ is the rightmost eigenvalue of $M$. Furthermore, we necessarily have that $\mu^g \in [-\lambda_1, \infty)$.
\end{theorem}
\begin{proof}
	See \cite[Theorem 3.2]{Adachi2017}.
\qed\end{proof}
In general, \eqref{eqn:trs} is guaranteed to possess a unique global optimizer, except perhaps in a special case that has deservedly been given a special name:
\spnewtheorem*{definition*}{Definition\!\!}{\bfseries}{\rmfamily}
\begin{definition*}[{Hard case}]
	Problem \eqref{eqn:trs} belongs to the \emph{hard case} when \mbox{$\mu^g = -\lambda_1$}.
\end{definition*}
We will see that in the hard case \eqref{eqn:trs} is still guaranteed to have a global minimizer, but it is not necessarily unique. Furthermore, the computation of the minimizer(s) can be more challenging than in the ``standard'' case.

Moreover, due to the nonconvexity of \eqref{eqn:trs} there can exist a local minimizer that is not global. We derive an analogous result for this minimizer that is called ``local-nonglobal'' \cite{Martinez1994}:
\begin{theorem} \label{thm:trs_local_min}
	Problem \eqref{eqn:trs} has a second-order sufficient local-nonglobal minimizer if and only if it does not belong to the hard case and the second-rightmost eigenvalue of $M$ \newcontentB{is real and simple}. Furthermore, if such a minimizer exists, then its Lagrange multiplier $\mu^{\ell}$ is equal to this second-rightmost eigenvalue.
\end{theorem}
\begin{proof}
	See \S \ref{sec:trs_proof}.
\qed\end{proof}

The Lagrange multipliers for the global and local-nonglobal minimizers of \eqref{eqn:trs} can therefore be identified by calculating the two rightmost eigenvalues\footnote{Except perhaps when the local-nonglobal minimizer is not second-order sufficient.} of $M$, which can be done efficiently e.g.\ with the Arnoldi method.

Furthermore, for each of the respective multipliers a corresponding minimizer can be extracted immediately  by the respective eigenvectors $[z_1^*;\; z_2^*]$ of $M$ using as
\begin{equation} \label{eqn:standard_extraction}
	x^* = -\text{sign}(q^T z_2^*) r \frac{z_1^*}{\norm{z_1^*}_2},
\end{equation}
unless $z_1^* = 0$. Indeed, this follows from the next Proposition:
\newcontentB{
\begin{proposition}
	Every eigenpair $(\mu^*, [z_1^*; \; z_2^*])$ of $M$ gives a KKT point $(\mu^*, x^*)$  for \eqref{eqn:trs} where $x^*$ is defined by \eqref{eqn:standard_extraction}, unless $z_1^* = 0$.
\end{proposition}
\begin{proof}
	This proof is based on \cite[\S3.3]{Adachi2017} that shows the result only for the rightmost eigenpair of $M$.

	We will first show the result for $y^* = \frac{-r^2}{q^T z_2^*}z_1^*$ and then show that $y^* = x^*$. Note that using $x^*$ is preferred over $y^*$ because it can avoid unnecessary numerical errors according to \cite[\S3.3]{Adachi2017}.  We also emphasize that all of the proof assumes that $z_1^*$ is nonzero.

	For $y^*$, we have to show that it is well defined, feasible and stationary. Regarding the first two points, note that \eqref{eqn:eigenvalue_problem} (first row) gives
	\begin{equation*}
		-P z_1^* + q q^Tz_2^*/r^2 = \mu^* z_1^* \Rightarrow
		(z_2^*)^T (P + \mu^* I)z_1^* = (z_2^*)^T q q^T z_2^*/r^2,
	\end{equation*}
	or since $(P + \mu^* I)z_2^* = z_1^*$ due to \eqref{eqn:eigenvalue_problem} (second row), we have $\norm{z_1^*}^2 = {(q^Tz_2^*/r)}^2$, which implies that $q^T z_2^* \neq 0$ because $z_1^* \neq 0$ by assumption. Thus:
	\begin{equation*}
		\norm{z_1^*}^2_2 = (q^Tz_2^*/r)^2 \Rightarrow \norm{-\frac{r^2}{q^T z_2^*}z_1^*}_2^2 = r^2 \Rightarrow \norm{y^*}_2 = r.
	\end{equation*}

	To show stationarity, note that the first row of \eqref{eqn:eigenvalue_problem} also gives:
	\begin{equation*}
		-P z_1^* + qq^T z_2^*/r^2 = \mu^* z_1^* \Rightarrow P \frac{-z_1^*r^2}{q^T z_2^*} + q + \mu^* \frac{-z_1^*r^2}{q^T z_2^*} = 0 \Rightarrow P y^* + q + \mu^* y^* = 0.
	\end{equation*}

	Finally, we show that $y^* = x^*$:
	\begin{equation*}
		y^* = y^* \frac{r} {\norm{y^*}_2} = \frac{-r^2z_1^*}{q^T z_2^*}  {\frac{|q^T z_2^*|}{\norm{-r^2 z_1^*}}}_2 r = -\text{sign}(q^T z_2^*) r \frac{z_1^*}{\norm{z_1^*}_2} = x^*.
	\end{equation*}
	\qed
\end{proof}
}

\paragraph{Dealing with the hard case:}
Extracting a solution $x^*$ via \eqref{eqn:standard_extraction} is not possible when $z_1^* = 0$, since \eqref{eqn:standard_extraction} is not well-defined in that case. \newcontentB{In this case, $(\mu^*, z_2^*)$ is an eigenpair of $-P$ due to \eqref{eqn:eigenvalue_problem}. Note that, according to \cite[Lemma 3.3]{Martinez1994},} this can never happen for the local-nonglobal minimizer; thus we can always extract the local-nonglobal minimizer with \eqref{eqn:standard_extraction}. However, it is possible for global minimizers in the hard case, i.e.\ the case where $\mu^* = -\lambda_1$. In the hard case, the necessary and sufficient conditions for global optimality are, due to the KKT conditions of \eqref{eqn:trs} and Theorem \eqref{thm:trs_global_min}, the following:
\begin{align}
	\label{eqn:hard_case_subspace}
	(P - \lambda_1 I) x + q = 0 \\
	\label{eqn:hard_case_norm}
	\norm{x}_2 = r.
\end{align}
Note that $P - \lambda_1 I$ is singular, and the above system of equations represents the intersection of an affine subspace with a sphere. This intersection must be non-empty, since \eqref{eqn:trs} necessarily has a global minimizer as it arises from the minimization of a smooth function over a compact subset of $\mathbb{R}^n$. One can then solve \eqref{eqn:hard_case_subspace}--\eqref{eqn:hard_case_norm}, by computing the minimum length solution $y_{\text{min}}$ of $(P - \lambda_1 I) x + q = 0$ and return $x^* = y_{\text{min}} + \alpha \upsilon$ where $\alpha$ is a scalar such that $\norm{x^*}_2 = \norm{y_{\text{min}} + \alpha \upsilon}_2 = r$ and $\upsilon$ is any null-vector of $P - \lambda_1I$. Interestingly, \newcontentB{when $z_1^*=0$, $z_2^*$ is a null vector of $P - \lambda_1\newcontent{I}$ due to \eqref{eqn:eigenvalue_problem}}, thus the only additional computation for extracting a solution when \newcontentB{$z_1^* = 0$} is finding the minimum-length solution of a symmetric linear system. This can be achieved e.g. with MINRES-QLP \cite{Choi2011}, or with the Conjugate Gradient method \cite[Theorem 4.3]{Adachi2017}.

\subsection{Equality-constrained Trust-Region Subproblems} \label{subsec:trs_full}
We will now extend the preceding results of this section to also account for the presence of linear equality constraints, $Ax = b$, in the TRS. We will show that, given an operator that projects an $n-$dimensional vector to the nullspace of $A$, we can  calculate global as well as local-nonlocal minima of \eqref{eqn:trs_full} by applying a projected Arnoldi method to $M$.

For simplicity, and without loss of generality\footnote{Problems with linear constraints of the form $Ax = b$ can be transformed to \eqref{eqn:trs_equality} via a change of variables $\tilde x = x + x_0$ where $x_0 \perp \mathcal{N}(A)$ such that $Ax_0 = b$, resulting in a quadratic cost and constraints: $A\tilde x = 0, \; \norm{\tilde x}_2 = \sqrt{r^2 - \norm{x_0}_2^2}$.}, we will assume that $b = 0$, thus solving the following problem:
\begin{equation}\label{eqn:trs_equality}
	\begin{array}{ll}
		\text{minimize}   & \frac{1}{2}x^T P x + q^T x \\
		\text{subject to} & \norm{x}_2 = r                         \\
											& Ax = 0
	\end{array}
\end{equation}
where $x \in \mathbb{R}^n$ is the decision variable, $P \in \mathbb{S}^{n}$, i.e.\ an $n \times n$ symmetric matrix, $A$ is a $m \times n$ matrix of full row rank and $q$ is an $n-$dimensional vector. Similarly to the preceding results for \eqref{eqn:trs}, we assume that $n - m> 1$ and $r > 0$.

In principle, \eqref{eqn:trs_equality} can be solved with the ``eigenproblem approach'' described in this section if we reduce \eqref{eqn:trs_equality} into a smaller TRS subproblem via a nullspace elimination procedure \cite[\S 15.3]{Nocedal2006}, obtaining
\begin{equation}\label{eqn:trust_region}
	\begin{array}{ll}
		\text{minimize}   & \frac{1}{2}y^T \tilde P y + \tilde q^T y\\
		\text{subject to} & \norm{y}_2 = r
	\end{array}
\end{equation}
where we define $Z$ to be a $n \times (n-m)$ orthonormal matrix with $\mathcal{R}(Z) = \mathcal{N}(A)$ and
\begin{align*}
	\tilde P \eqdef Z^T P Z, \quad \tilde q \eqdef Z^Tq.
\end{align*}
According to the preceding results of this section, global and local-nonglobal minimizer(s) $y^*$ of \eqref{eqn:trust_region} can be computed by calculating the two rightmost eigenvectors of $$\tilde M \eqdef \begin{bmatrix}
	-\tilde P & \tilde q \tilde q^T/r^2    \\
	I & -\tilde P
\end{bmatrix}.$$ The respective solution(s) $x^*$ of \eqref{eqn:main_problem} can then be recovered as
\begin{equation} \label{eqn:x-y}
	x^* = Zy^*.
\end{equation}
Some disadvantages of this approach are that a nullspace basis is required and that the structure or sparsity of the problem might be destroyed.

We will present an alternative method that avoids these issues, mirroring standard approaches for solving equality constrained quadratic programs \cite{Gould2001}. To this end, note that
\begin{equation*}
	\tilde M =
	\begin{bmatrix}
		-\tilde P & \tilde q \tilde q^T/ r^2 \\
		I & -\tilde P
	\end{bmatrix}
	=
	\begin{bmatrix}
		Z^T & 0\\
		0 & Z^T\\
	\end{bmatrix}
	\begin{bmatrix}
		-P & q q^T / r^2 \\
		I & -P
	\end{bmatrix}
	\begin{bmatrix}
		Z & 0\\
		0 & Z\\
	\end{bmatrix}.
\end{equation*}
Since Z is an orthonormal matrix, every eigenvalue-eigenvector pair $\{\mu, [\tilde z_1;\; \tilde z_2] \}$ of $\tilde M$ corresponds to an eigenvalue-eigenvector pair $\{\mu, [z_1;\; z_2] \}$ of the matrix
\begin{equation} \label{eqn:arnoldi_matrix}
	\begin{bmatrix}
		ZZ^T & 0\\
		0 & ZZ^T\\
	\end{bmatrix}
	\underbrace{
	\begin{bmatrix}
		-P &  q  q^T/ r^2 \\
		I & -P
	\end{bmatrix}}_{=M}
	\begin{bmatrix}
		ZZ^T & 0\\
		0 & ZZ^T\\
	\end{bmatrix}
\end{equation}
with $Z [\tilde z_1;\; \tilde z_2] = [z_1;\; z_2]$. Note that \eqref{eqn:arnoldi_matrix} also has an extra eigenvalue at zero with multiplicity \newcontent{$2m$}.

Recalling that $Z$ is an orthonormal matrix spanning $\mathcal{N}(A)$, we note that multiplication with $ZZ^T$ is equivalent to projection onto the nullspace of $A$. In practice, eigenvalues of \eqref{eqn:arnoldi_matrix} can therefore be calculated using the Arnoldi method on $M$ where we project the starting vector and every requested matrix-vector product with $M$ onto the nullspace of $\begin{bmatrix}A & 0 \\ 0& A \end{bmatrix}$.

Note that the two rightmost eigenvalues of $\tilde M$ correspond to the two rightmost eigenvalues of \eqref{eqn:arnoldi_matrix} if at least two eigenvalues of $\tilde M$ have positive real parts. We will show that we can always transform \eqref{eqn:trs_equality} so that this holds. Indeed if we shift $P$ to $P - \alpha I$ with an appropriate $\alpha \in \mathbb{R}$ so that $\tilde P$ become negative definite\footnote{Such an $\alpha$ can be obtained by applying e.g., the Lanczos method, or the Gershgorin circle Theorem on $P$.}, then $\tilde M$ has at least two eigenvalues with positive real part:
\begin{lemma}
	If $\tilde P \in \mathbb{S}^{n - m}$ is negative definite then $\tilde M$ has at least two eigenvalues with positive real part, unless we trivially have $n - m = 1$.
\end{lemma}
\begin{proof}
	To avoid redefining the notation used in the introduction of \S2 we will show the equivalent result for $(P, M)$, i.e.\ that $M$ has at least two eigenvalues with positive real part when $P$ is negative definite and $n > 1$.

	\newcontent{
   To do so, it suffices to show that $s$ has at most one root with non-positive real part. This is because $M$ has $2n \geq 4$ eigenvalues and the spectrum of $M$ matches the roots of $s$ on the left-hand complex plane (shown in Proposition \ref{prop:spectrum_new} that we will prove on \S \ref{sec:trs_proof}, given that $P$ is negative definite).}

	To show this, note that $s$ has at most one (simple) real root with \newcontent{nonpositive} real part since
	 \begin{equation} \label{eqn:s_dot}
		s'(\newcontentB{\alpha}) = -2\sum_{j=1}^{n}\frac{(w_j^Tq)^2}{(\lambda_j + \alpha)^3}
	\end{equation}
	is \newcontentB{positive} for any \newcontent{$\alpha \leq 0$} since $\lambda_1, \dots, \lambda_n < 0$ by assumption. To conclude the proof, it suffices to show that all the roots of $s$ with \newcontent{nonpositive} real part are real. Indeed, note that for any $a, b \in \mathbb{R}$ we have
	\begin{align*}
		\text{Im}(s(a + bi)) = -\sum_{j=1}^{n}\frac{2(\lambda_j + a)b}{((\lambda_j + a)^2 + b^2)^2}(w_j^Tq)^2,
	\end{align*}
	where $i$ is the imaginary unit. A detailed derivation of the above equation is provided in Lemma \ref{lem:aux}. For any \newcontent{$a
	\leq 0$} we have \newcontent{$-a \geq 0 > \lambda_1, \dots, \lambda_n \Rightarrow 2(\lambda_j + a) < 0$}, $j = 1, \dots, n$ by assumption. Thus $\text{Im}(s(a + bi)) \neq 0 \Rightarrow s(a + bi) \neq 0$ for any \newcontent{$a \leq 0$} unless $b = 0$.
\qed\end{proof}
Such a shift would not affect the (local/global) optimizers of \eqref{eqn:main_problem} as
\begin{align*}
	\frac{1}{2}x^T (P - \alpha I) x + q^T x
	&= \frac{1}{2}x^T P x + q^T x - \frac{\alpha}{2} ||x||_2^2 \\
	\intertext{or, since $\norm{x}_2 = r$,}
	\frac{1}{2}x^T (P - \alpha I) x + q^T x
	&= \frac{1}{2}x^T P x + q^T x - \frac{\alpha}{2} r^2,
\end{align*}
where $\alpha \in \mathbb{R}$. Thus, we can always transform \eqref{eqn:trs_equality} so that the two rightmost eigenvalues of $\tilde M$ are equal to the two rightmost eigenvalues of \eqref{eqn:arnoldi_matrix}.

Having computed the two rightmost eigenvalues/vectors with a projected Arnoldi method applied to $M$, global/local solutions of \eqref{eqn:main_problem} can be obtained from the respective eigenvectors. Indeed, recall that according to \eqref{eqn:standard_extraction} and \eqref{eqn:x-y}, an appropriate (rightmost/second-rightmost) eigenvector $[\tilde z_1^*;\; \tilde z_2^*]$ of \eqref{eqn:arnoldi_matrix} gives a solution $x^*$ to \eqref{eqn:main_problem} as follows
\begin{align} \label{eqn:standard_extraction_full}
	x^* &= Zy^* = -Z\text{sign}(\tilde q^T  \tilde z_2^*)r\frac{\tilde z_1^*}{\norm{\tilde z_1^*}_2} = -\text{sign}(q^T (Z \tilde z_2^*))r\frac{Z \tilde z_1^*}{\norm{Z\tilde z_1^*}_2}\\
	&= -\text{sign}(q^T  z_2^*)r\frac{z_1^*}{\norm{z_1^*}_2}
\end{align}
where $[z_1^*;\; z_2^*]$ is an eigenvector of $M$. Thus, $x^*$ can be extracted solely by the eigenvector of $M$ unless $z_1^* = 0$.

\paragraph{Dealing with the hard case:} \newcontentB{In the case where $z_1^* = 0$, which can appear in the hard case,} one would have to calculate the minimum length solution of
\begin{equation} \label{eqn:subspace_hardcase_constraints}
	(\tilde P + \mu^* I) y + \tilde q = 0 \\
\end{equation}
and then construct a global solution $x^*$ as $x^* = Z(y + \alpha \tilde z_2^*) = Zy + \alpha z_2^*$, where $\alpha$ is a scalar such that $\norm{x^*}_2 = r$. However, the minimum length solution of \eqref{eqn:subspace_hardcase_constraints} can be obtained by the minimum length solution of
\begin{align} \label{eqn:hard-case-characterization}
(P + \mu^* I) x +  q = 0 \\
	Ax = 0
\end{align}
which can be calculated via projected MINRES-QLP or a projected CG algorithm \cite[Theorem 4.3]{Adachi2017} \cite{Gould2001}.

\newcontentB{
\subsection{An algorithm for computing global and local solutions of the TRS} \label{subsec:trs_algorithm}
Using the ideas presented in this section, we now present Algorithm \ref{alg:trs}, a practical algorithm for computing global and second-order sufficient local-nonglobal solutions of equality constrained trust-region subproblems. In the remainder of this section, we establish the correctness of Algorithm \ref{alg:trs}.
}
\begin{algorithm}
	\textbf{Given} the problem data $P \in \mathbb{S}^n$, $q \in \mathbb{R}^n$, $r \in \mathbb{R}_+$, $A \in \mathbb{R}^{m \times n}$ where $A$ is full row rank with $m < n - 1$ and $\Pi : \mathbb{C}^n \mapsto \mathbb{C}^n$, a projector into the nullspace of $A$\;
	Shift $P$ by a multiple of identity so that it becomes negative definite\;
	Run an Arnoldi method on
	\begin{equation*}
		\Pi \left(
		\begin{bmatrix}
			- P &  q  q^T/  r^2    \\
			I & -P
		\end{bmatrix}
		\begin{bmatrix}
			z_1 \\
			z_2
		\end{bmatrix}
		\right)
		= \mu
		\begin{bmatrix}
			z_1 \\
			z_2
		\end{bmatrix},
	\end{equation*}
	where $\Pi\left(\begin{bmatrix}
		x_1 \\
		x_2
	\end{bmatrix}\right) \eqdef \begin{bmatrix}
		\Pi(x_1) \\
		\Pi(x_2)
	\end{bmatrix}$ for any $x_1, x_2 \in \mathbb{C}^n$, starting from an initial point $\Pi(z) \in \mathbb{C}^{2n}$, to calculate a rightmost eigenvalue/vector pair $(\mu^g, z^g)$\;

	\uIf{$z_1^g \neq 0$}{
			$x^g \leftarrow -\text{sign}(q^T z_2^g) r \frac{z_1^g}{\norm{z_1^g}_2}$\;
			Resume the projected Arnoldi method to compute a second-rightmost eigenvalue/vector $(\mu^{\ell}, z^{\ell})$\;
			\uIf{$\mu^{\ell}$ \newcontentB{is real and simple}}{
				$x^{\ell} \leftarrow -\text{sign}(q^T z_2^{\ell}) r \frac{z_1^{\ell}}{\norm{z_1^{\ell}}_2}$\;
				\textbf{Return} $x^g$ and $x^{\ell}$ as the unique global and local-nonglobal minimizers\;
			}
			\Else{
				\textbf{Return} $x^g$ as a global minimizer and state that no second-order sufficient local-nonglobal minimizer exists.\;
			}
	}
	\Else{
		Compute the minimum length solution $x_{\min}$ of
		\begin{align*}
			(P + \mu^g I) x + q = 0 \\
			Ax = 0
		\end{align*} via projected MINRES-QLP or projected CG \cite[Theorem 4.3]{Adachi2017}\;
		Solve a quadratic equation to find $\alpha_1, \alpha_2$ such that $\norm{x_{\min} + \alpha_{1/2} z^g_2}_2 = r$\;
		\textbf{Return} $x_{\min} + \alpha_{1} z^g_2$, $x_{\min} + \alpha_{2} z^g_2$ as global solutions, state that no local-nonglobal minimizer exists\;
	}
	\caption{Calculation of global and local-nonglobal minimizers of the equality constrained trust-region subproblem \eqref{eqn:trs_equality}}
	\label{alg:trs}
\end{algorithm}

\newcontentB{
	Algorithm \ref{alg:trs} uses the results of this section (in particular Theorems \ref{thm:trs_global_min} and \ref{thm:trs_local_min}, Equation \eqref{eqn:standard_extraction_full} and the paragraph ``dealing with the hard case'' of \S \ref{subsec:trs_full}) to compute the minimizers of \eqref{eqn:trs_equality}. The only novel point is the way we detect the presence of the local-nonglobal minimisers, which we proceed to show to be valid.

	If the TRS is in the hard case, then the algorithm should only return global minimiser(s). Indeed, this is the case as
	\begin{enumerate}
		\item If $z_1^g=0$, then the algorithm terminates at line 16, according to the remarks of the paragraph ``Dealing with the hard case'' of \S \ref{subsec:trs_full}.
		\item If $z_1^g\neq0$, then there must exist another eigenvector associated with $\mu^g$ besides $[z_1^g; \; z_2^g]$. This is because $\mu^g = -\lambda_1$ thus according to Proposition \ref{prop:spectrum_new} of \S\ref{sec:trs_proof} there exists a vector $u \in \mathcal{N}(P - \lambda_1 I)$ that is orthogonal to $q$ which makes $(\mu^g, [0;\; u])$ an eigenpair of $M$ due to \eqref{eqn:eigenvalue_problem}. \newline Thus $\mu^g=\mu^l$, which implies that $\mu^l$ is not simple. As a result the algorithm will only return a global minimizer using \eqref{eqn:standard_extraction} at line 11.
	\end{enumerate}

	If the TRS is not in the hard case, then we must have $z_1^g \neq 0$ (see remarks of paragraph ``Dealing with the hard case'' of \S \ref{sec:trs}). Thus the algorithm will return either in lines 9 or 11, detecting correctly the presence of the local-nonglobal minimiser according to Theorem \ref{thm:trs_local_min}.

Finally, we emphasize that in the hard case the TRS can have an infinite number of solutions, but only one or two of them are returned by Algorithm \ref{alg:trs}.
}
\renewcommand{\thesubsection}{2.A}
\subsection{Proof of Theorem \ref{thm:trs_local_min}} \label{sec:trs_proof}
\renewcommand{\thesubsection}{\arabic{section}.\arabic{subsection}}
In this subsection we provide a proof of our main theoretical result, Theorem \ref{thm:trs_local_min}, which shows that \eqref{eqn:trs} has a second-order sufficient local-nonglobal minimizer if and only if it does not belong to the hard case and the second-rightmost eigenvalue of the matrix
\begin{equation*}
	M \eqdef
	\begin{bmatrix}
		-P & q q^T/r^2    \\
		I & -P
	\end{bmatrix},
\end{equation*}
with counted algebraic multiplicities, is real and simple. Furthermore, we show that if such a minimizer exists, its Lagrange multiplier $\mu^{\ell}$ is equal to this second-rightmost eigenvalue.


The proof is based on the following Lemma, which follows from \cite{Martinez1994}:
\begin{lemma} \label{lem:martinez}
	A KKT point $x^{\ell}$ of \eqref{eqn:trs} with Lagrange multiplier $\mu^{\ell}$ is a second-order sufficient local-nonglobal minimizer iff $q \not \perp w_1$, $\mu^{\ell} \in (-\lambda_1, -\lambda_2)$ and $s'(\mu^{\ell}) > 0$.
\end{lemma}
\begin{proof}
	See \cite[Theorem 3.1]{Martinez1994}.
\qed\end{proof}

Before we use Lemma \ref{lem:martinez} we will need the following two ancillary results:
\begin{lemma} \label{lem:aux}
	For any complex number $a + bi \in \mathbb{C}$ where $a, b$ are real numbers such that $a$ is in the domain of $s$ and $b \neq 0$, we have:
	\begin{equation}
		\mathrm{Re}\left(s(a + bi)\right) < s(a).
	\end{equation}
\end{lemma}
\begin{proof}
	Recall the definition of $s$:
	\begin{align*}
		s(a + bi) = \sum_{j=1}^{n}\frac{(w_j^T q)^2}{(\lambda_j + a + \newcontent{bi})^2} - r^2
	\end{align*}
	and define, for convenience, $\kappa_j = w_j^T q$, $j = 1 ,\dots, n$.
	Thus
	\begin{align*}
		s(a + bi)
		&= -r^2 + \sum_{j=1}^{n}\frac{\kappa_j^2}{(\lambda_j + a + bi)^2}\\
		&= -r^2 + \sum_{j=1}^{n}\kappa_j^2\frac{(\lambda_j + a)^2 - b^2}{((\lambda_j + a)^2 + b^2)^2} - \kappa_j^2\frac{2(\lambda_j + a)b}{((\lambda_j + a)^2 + b^2)^2}i, \\
	\intertext{hence}
		\mathrm{Re}(s(a + bi)) &= -r^2 + \sum_{j=1}^{n}\kappa_j^2\frac{(\lambda_j + a)^2 - b^2}{((\lambda_j + a)^2 + b^2)^2} \\
		&< -r^2 + \sum_{j=1}^{n}\kappa_j^2\frac{(\lambda_j + a)^2}{((\lambda_j + a)^2)^2},
	\end{align*}
	or, equivalently, $\mathrm{Re}(s(a + bi)) < s(a)$.
\qed\end{proof}
\begin{lemma} \label{lem:cauchy}
	If $s(x) < 0$ for some real $x$ in the domain of $s: \mathbb{C} \mapsto \mathbb{C}$ then $s$ has an equal number of poles and zeros with real parts in $[x, \infty)$.
\end{lemma}
\begin{proof}
	We will prove this result with the argument principle. We propose encircling the half-space
	\begin{equation*}
		S(x) \eqdef \set{c \in \mathbb{C}}{c = a + bi, \text{ where } a \geq x, b \in \mathbb{R}}
	\end{equation*}
	with the contour of Figure \ref{fig:encirclement} and show that it maps to a contour that lies strictly in the left half plane. The suggested closed contour $C$ intersects the real axis at $x$ and at $\infty$. It consists of two segments: $C_1 = \{\omega + \beta i \mid \beta \in \mathbb{R} \}$, i.e.\ a segment that is parallel to the imaginary axis, and $C_2$, a half-circle with infinite radius.
\begin{figure}[t]
	\centering
	\includegraphics[width=.4\textwidth]{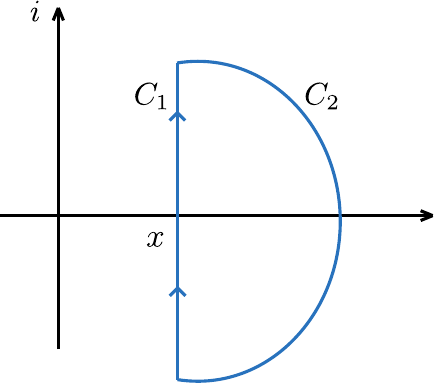}
	\caption{Suggested closed contour for encircling $S(x)$.}
	\label{fig:encirclement}
\end{figure}

\newcontent{By the argument principle~\cite{Ahlfors1966}
we have $\frac{1}{2\pi i}\int_C\frac{s'(z)}{s(z)}dz=Z-P$, where $Z,P$
are the number of roots and poles of $s$ in the domain enclosed by $C$. It follows that the number of poles in $S$ is equal to the number of zeros in $S$ if $\int_C\frac{s'(z)}{s(z)}dz=0$, which holds if}
there are no poles or zeros of $s$ \newcontent{on} $C$ and $s(C)$ does not enclose the origin. By examining \eqref{eqn:secular} we conclude that $s$ has only real, finite poles. Thus, the contour $C$ does not intersect any pole of $s$. It remains to show that there exist no zeros of $s$ \newcontent{on} $C$ and $s(C)$ does not enclose the origin. We show both by proving that every point in $s(C) = s(C_1) \cup s(C_2)$ has negative real part. For $s(C_2)$, it is easy to see that, since $s(C_2) = \{ -r^2\}$ as $\lim_{|w| \to \infty} s(w) = -r^2$. Regarding $s(C_1)$ note that $s(x) < 0$, and thus Lemma \ref{lem:aux} gives $\mathrm{Re}(s(x + \beta i)) \leq s(x) < 0 \; \forall \beta \in \mathbb{R}$, i.e.,\ any point in $s(C_1)$ has negative real part.
\qed\end{proof}
We will now connect the second-rightmost root of the secular equation $s$ with the results of Lemma \ref{lem:martinez}. Note that, due to Lemma \ref{lem:aux}, if $s$ has  a real root $\mu$ then it cannot also have \newcontent{an unreal} root with real part $\mu$ and vice versa. Thus, the concept of ``the second-rightmost root of $s$'' is well defined as long as we assume (or are prepared to check) that this root is real.
\begin{lemma} \label{lem:argument_principle} \hfill
	\begin{enumerate}[(i)]
		\item Suppose that $q \not \perp w_1$ and $\mu \in (-\lambda_2, -\lambda_1)$ is \newcontent{a} simple root of $s$ with $s'(\mu) > 0$. Then, $\mu$ is the second-rightmost root of $s$ in $\mathbb{C}$.
		\item Suppose that $q \not \perp w_1, w_2$, $\lambda_1 \neq \lambda_2$, $\mu \in \mathbb{R}$ is the second-rightmost root of $s$ in $\mathbb{C}$ and $\mu$ is simple. Then, $\mu \geq -\lambda_2$.
	\end{enumerate}
\end{lemma}
\begin{proof}
	Note that in both cases $q \not \perp w_1$, thus \eqref{eqn:trs} does not belong to the hard case \cite{Adachi2017}. Hence $s$ has a real root in $(-\lambda_1, \infty)$ (Theorem \ref{thm:trs_global_min}) which we denote by $\mu^g$. Moreover, $\mu^g$ is simple because $s'(\mu^g) = -2\sum_{j=1}^{n}(w_j^Tq)^2/(\lambda_j + \mu^g)^3$ is negative since $\mu^g > -\lambda_1 > \dots > -\lambda_n$.

	In the subsequent analysis, we will frequently refer to the poles of $s$ which are
	\begin{equation*}
		\set{-\lambda_i}{i = 1, \dots, n \text{ such that } q^Tw_i \neq 0},
	\end{equation*}
	and have double multiplicity, and the following parametric set
	\begin{equation*}
		S(x) \eqdef \set{c \in \mathbb{C}}{c = a + bi, \text{ where } a \geq x, b \in \mathbb{R}}
	\end{equation*}
	defined for any $x \in \mathbb{R}$.

	We will begin with claim (i). Since $\mu < -\lambda_1 < \mu^g$ we conclude that $\mu$ is not the rightmost root of $s$. In order to show that it is the \emph{second} rightmost, it suffices to show that there are \newcontentB{at most} two roots of $s$ in $S(\mu)$.

	Since $s(\mu) = 0$ and $s'(\mu) > 0$, we have $s(\mu - \epsilon) < 0$ for a sufficiently small $\epsilon > 0$. According to Lemma \ref{lem:cauchy}, $s$ has equal number of zeros and poles in $S(\mu - \epsilon)$. By assumption $q \not \perp w_1$ and $\mu > -\lambda_2$, thus $-\lambda_1$ is the only (double) pole of $s$ in $S(\mu - \epsilon)$. Hence, the number of zeros in $ S(\mu - \epsilon) \newcontentB{\supset} S(\mu)$ must be two. This concludes the proof of claim (i).

	We will now prove claim (ii). Assume the contrary, i.e.\ that $\mu < -\lambda_2$. Consider first the case where $s'(\mu) > 0$. Since $s(\mu) = 0$ we can choose a sufficiently small $\epsilon > 0$ such that $s(\mu - \epsilon) < 0$. Thus $s$ has equal number of zeros and poles in $S(\mu - \epsilon)$  (Lemma \ref{lem:cauchy}). By assumption $\mu < -\lambda_2$ and $q \not \perp w_1, w_2$, and thus the double poles $-\lambda_1$ and $-\lambda_2$ are in $S(\mu - \epsilon)$. Thus $s$ has (at least) four zeros in $S(\mu - \epsilon)$, and since $\epsilon$ can be arbitrarily small, the same holds for $S(\mu)$.

	We conclude that there should be extra roots for $s$ in $S(\mu)$, besides the simple roots $\mu^g$ and $\mu$. Moreover, these extra roots cannot have real part equal to $\mu$ as according to Lemma \ref{lem:aux} $\newcontent{\mathrm{Re}(s(\mu + b i))} < s(\mu) = 0$ for any real $b \neq 0$. Thus their real parts must be in $(\mu, \infty)$. This means that $\mu$ is not the second-rightmost root of $s$, which is a contradiction.

	Finally, if $s'(\mu) < 0$, then, using the same arguments as before, we conclude that there exist four roots in $S(\mu + \epsilon)$ for any sufficiently small $\epsilon > 0$. Thus, there exist extra roots (besides the simple root $\mu^g$) with real part in $(\mu, \infty)$, which, again, contradicts the assumption that $\mu$ is the second-rightmost root of $s$.
\qed\end{proof}
Lemma \ref{lem:argument_principle}, in combination with Lemma \ref{lem:martinez}, provides an ``almost'' iff relationship between the second-rightmost root of $s$ and the local-nonglobal minimizer of \eqref{eqn:trs}. However, the quantity of primary interest is the second-rightmost eigenvalue of $M$ instead of the second-rightmost root of $s$. The following result characterizes the spectrum of $M$ and its tight connection with the roots of $s$:
\newcontent{
\begin{proposition} \label{prop:spectrum_new}
The spectrum of $M$ consists of:
\begin{itemize}
	\item the roots of $s$ with matched algebraic multiplicity; and
	\item all of the eigenvalues $-\lambda_i$ of $-P$ that are non-simple \newcontentB{for $-P$} or have $w_i$ (the $i-$th eigenvector of $P$) orthogonal to $q$.
\end{itemize}
\newcontentB{Furthermore, every $-\lambda_i$ that is in the spectrum of $M$ is non-simple for $M$.}
\end{proposition}
\begin{proof}
	We will prove the result by deriving a closed form expression for the characteristic polynomial of $M$. Since $-I$ and  $P + \mu I$ commute, we have \cite[Theorem 3]{Silvester2000}:
\begin{align*}
	\det(\mu I - M) = &
	\det\left(
		\begin{bmatrix}
			P + \mu I & -qq^T/r^2 \\
			-I & P + \mu I
		\end{bmatrix}
	\right)
	= \det( (P + \mu I)^2 - q q^T/r^2) \\
	\intertext{or, using the matrix determinant lemma,}
	\det(\mu I - M) &= -r^{-2}\left(q^T (P + \mu I)^{-2}q - r^2 \right) \det((P + \mu I)^2)
\end{align*}
for all $\mu \neq -\lambda_1, \dots -\lambda_n$. Thus, recalling \eqref{eqn:stationarity_trs}-\eqref{eqn:secular} we have
\begin{align}
	det(\mu I - M) &=
	-r^{-2} \left(\sum_{i=1}^{n}\frac{(w_i^T q)^2}{(\lambda_i + \mu)^2} - r^2\right)\prod_{i = 1}^n(\lambda_i + \mu)^2
	\label{eqn:characteristic_polynomial_prequel}
	\\
	&= -r^{-2}\sum_{i = 1}^{n} (w_i^T q)^2 \prod_{j \neq i} (\lambda_j + \mu)^2 + \prod_{i = 1}^{n} (\lambda_i + \mu)^{\newcontentB{2}}, \label{eqn:characteristic_polynomial}
\end{align}
for all $\mu \neq -\lambda_1, \dots -\lambda_n$. Thus, \eqref{eqn:characteristic_polynomial} and the characteristic polynomial of $M$ coincide in all $\mathbb{C}$ except perhaps on $2n$ points. It follows from continuity arguments that \eqref{eqn:characteristic_polynomial} is in fact the characteristic polynomial of $M$.

By examining \eqref{eqn:characteristic_polynomial_prequel} (2nd leftmost) and \eqref{eqn:characteristic_polynomial}, we conclude that the eigenvalues of $M$ include all the roots of $s$ with matched algebraic multiplicity. The secular equation $s$ has $2n$ roots whenever the eigenvalues of $P$ are distinct and $w_i^T q \neq 0$ for all $i = 1, \dots, n$. Otherwise, i.e. for every $i$ with $\lambda_i$ non-simple or $w_i^T q = 0$, $M \in \mathbb{R}^{2n \times 2n}$ has extra \newcontentB{non-simple} eigenvalues at these $-\lambda_i$, which constitute the only differences between the spectrum of $M$ and the roots of~$s$.
\qed
\end{proof}
}
\newcontentB{
The following result follows directly from the above Proposition and will be useful for the rest of this section.
\begin{corollary} \label{cor}
	No simple eigenvalue of $M$ is in the spectrum of $-P$.
\end{corollary}
}

We are now ready to prove the main result of this section\newcontent{, Theorem 2. The proof will occupy the rest of this subsection.}
Note, again, that, due to Lemma \ref{lem:aux} and Proposition \ref{prop:spectrum_new}, if $M$ has a real eigenvalue $\mu$ then it cannot also have \newcontent{an unreal} eigenvalue with real part $\mu$ and vice versa. Thus, the concept of ``the second-rightmost eigenvalue of $M$'' is well defined as long as we assume (or are prepared to check) that this eigenvalue is real.

\underline{``only-if''}
In this case we know that $\mu^{\ell}$ is the Lagrange multiplier of a second-order sufficient local-nonglobal minimizer and we have to show that
\begin{enumerate}[(i)]
	\item \eqref{eqn:trs} is not in the hard case.
	\item $\mu^{\ell}$ is a simple eigenvalue of $M$.
	\item $\mu^{\ell}$ is the second-rightmost eigenvalue of $M$.
\end{enumerate}
Points \newcontentB{(i), and (ii)} follow directly from Lemma \ref{lem:martinez} and Proposition \ref{prop:spectrum_new} (note that $q \not \perp w_1$ implies that \eqref{eqn:trs} is not in the hard case \cite{Adachi2017}). In order to prove \newcontentB{(iii)} we will first show that the spectrum (counted with algebraic multiplicities) of $M$ with real part larger than $-\lambda_2$ coincides with the roots of $s$ in the same region. \newcontent{This follows from Proposition \ref{prop:spectrum_new}, because $\lambda_1$ is simple (due to $\mu^{\ell} \in (-\lambda_2, -\lambda_1)$) and  $q^T w_1 \neq 0$. Point \newcontentB{(iii)} then follows from Lemma \ref{lem:argument_principle}(i) which shows that $\mu^{\ell}$ is the second-rightmost root of $s$, and thus the second-rightmost eigenvalue of $M$.}

\underline{``if''} In this case, we know that \eqref{eqn:trs} is not in the hard case and that $\mu^{\ell}$ is the second-rightmost eigenvalue of $M$ which is real and simple, and we want to show that $\mu^{\ell}$ is the Lagrange multiplier of the second-order sufficient local-nonglobal minimizer of \eqref{eqn:trs}.

Note that, by assumption, $\mu^{\ell}$ is in the spectrum of $M$ \newcontentB{but, due to Corollary \ref{cor}, not in that of $-P$}. Thus $\mu^{\ell}$ is a root of $s$ due to \newcontent{Proposition \ref{prop:spectrum_new}}. As a result, it suffices to show that $\mu^{\ell}$ is in $(-\lambda_2, -\lambda_1)$ and $s'(\mu^{\ell}) > 0$ (Lemma \ref{lem:martinez}).

We will first show that $\mu^{\ell} \in (-\lambda_2, -\lambda_1)$. Since we are not in the hard case, $M$ has a real eigenvalue in $(-\lambda_1, \infty)$ (Theorem \ref{thm:trs_global_min}) which is simple and unique in $(-\lambda_1, \infty)$ due to \newcontent{Proposition \ref{prop:spectrum_new}} and the fact that $s'(\mu) = -2\sum_{j=1}^{n}(w_j^Tq)^2/(\lambda_j + \mu)^3$ is negative in that region. Thus, $-\lambda_1$ is not in the spectrum of $M$ as otherwise it would be its second-rightmost eigenvalue (which by assumption is real) and it \newcontentB{would imply that $\mu^{\ell}$ is in the spectrum of $-P$, which we have already excluded}. Thus, $\mu^{\ell} \in (-\infty, -\lambda_1)$\newcontentB{, and $\mu^{\ell} \neq -\lambda_2$.}

It remains to show that $\mu^{\ell} \geq -\lambda_2$. \newcontent{If $q^T w_2 = 0$ then, $(-\lambda_2, [0; w_2])$ is an eigenpair of $M$, thus we must have $\mu^{\ell} > -\lambda_2$ so as to avoid $\mu^{\ell}$ (the second-rightmost eigenvalue of $M$) being in the spectrum of $-P$. Consider now the case where $q^T w_2 \neq 0$. Note that $\lambda_1 \neq \lambda_2$ and $q \not \perp w_1$ as otherwise $-\lambda_1$ is in the spectrum of $M$ (Proposition \ref{prop:spectrum_new}), which, according to the paragraph above is a contradiction. Thus, $\lambda_1 \neq \lambda_2$ and $q \not \perp w_1, w_2$ resulting in $\mu^{\ell} \geq -\lambda_2$ due to Lemma \ref{lem:argument_principle}(ii).}

Finally we show that $s'(\mu^{\ell}) > 0$, thereby concluding the proof. Note that $s'(\mu^{\ell}) \neq 0$ since by assumption $\mu^{\ell}$ is a simple root of $s$. Assume the contrary, i.e.\ $s'(\mu^{\ell}) < 0$. Since $s$ is convex in $(-\lambda_2, -\lambda_1)$ \cite[(3.12)]{Martinez1994} and $\lim_{\mu \to -\lambda_1}s(\mu) = \infty$ we conclude that there \newcontentB{must exist another root} of $s$ in $(\mu^{\ell}, -\lambda_1)$. This contradicts the assumption that $\mu^{\ell}$ is the second-rightmost root of $s$.

%% file: sections/algorithm.tex
\section{An active-set algorithm for \eqref{eqn:main_problem}} \label{sec:algorithm}
We are now in a position to present the main contribution of this paper, an active-set algorithm for solving \eqref{eqn:main_problem}. We will first present an algorithm for the special that $r_{\min} = r_{\max} \eqdef r$ and then describe how to generalize for any $r_{\min}, r_{\max}$.
\subsection{Solving \eqref{eqn:main_problem} when $r_{\min} = r_{\max} \eqdef r$} \label{sec:active_set}
In this section we introduce a primal active-set approach for the optimization of \eqref{eqn:main_problem} when $r_{\min} = r_{\max} \eqdef r$. It will be useful in the subsequent analysis to recall the Karush-Kuhn-Tucker (KKT) conditions of \eqref{eqn:main_problem}, which are
\begin{align}
	\label{eqn:KKT1}
	\nabla f(x) + A^T\kappa + \mu x &= 0 \\
	\label{eqn:KKT2}
	\kappa                     & \geq 0                                                 \\
	\label{eqn:KKT3}
	\kappa_i( a_i^T x - b_i)   & = 0, \quad i = 1, \dots, m  \\
	\label{eqn:KKT4}
	A x                         & \leq b, \quad  \quad x^Tx = r^2.
\end{align}

As is typical from a primal active-set approach, our algorithm starts from a given feasible point of \eqref{eqn:main_problem} and generates iterates that remain feasible for \eqref{eqn:main_problem} and have non-increasing objective values. At every iteration we treat a subset of the inequality constraints $Ax \leq b$ as equalities. We refer to this subset as the \emph{working set} $\mathcal{W}_k$ and define
\[
	\bar A_k = [a_i]_{i \in \mathcal{W}_k}, \quad \bar b_k = [b_i]_{i \in \mathcal{W}_k},
\]
where $[\cdot]$ denotes vertical (row-wise) concatenation.  To simplify the analysis, we will assume that  one of the simplest and most common constraint qualification holds for every iterate of the algorithm:
\begin{assumption} Linear Independence Constraint Qualification [LICQ] \\
  \label{ass:licq}
	The LICQ holds for every iterate of the suggested Algorithm, i.e.\ $\begin{bmatrix} \bar A_k \\ \bar x_k^T \end{bmatrix}$ is full row rank.
\end{assumption}

We will first give a brief, schematic description of our algorithm. At every iteration $k$ of our active-set based minimization procedure, we use Algorithm \ref{alg:trs} to compute minimizers of the subproblem
\begin{equation} \label{eqn:working_set_problem} \tag{SP}
	\begin{array}{ll}
		\text{minimize}   & \frac{1}{2}x^T P x + q^T x \\
		\text{subject to} & \norm{x}_2 = r    \\
											& \bar A_k x = \bar b_k.
	\end{array}
\end{equation}
We then attempt to move towards those minimizers in that hope that we either (i) ``hit'' a new constraint or (ii) set $x_{k + 1}$ to a \newcontentB{(potentially local)} minimizer of \eqref{eqn:working_set_problem}. Unfortunately due to the nonconvexity of the problem, we might achieve neither (i) nor (ii) by simply moving towards the minimizer of \eqref{eqn:working_set_problem}\newcontent{, as the objective of \eqref{eqn:main_problem} is not guaranteed to be non-increasing along these moves}. In that case, which we will show can happen only when \eqref{eqn:working_set_problem} has a unique minimizer, we perform a projected gradient descent followed (if necessary) by a move towards the global minimizer of \eqref{eqn:working_set_problem} that is guaranteed to achieve either (i) or (ii)\footnote{Technically, in this case $x_{k+1}$ will be a second order necessary stationary point of \eqref{eqn:working_set_problem}.}. Finally, at the end of each iteration, we check for termination and update the working set.

Our overall active-set procedure is shown in Algorithm \ref{alg:active_set}. We now proceed in the following subsections that describe in detail each of the Algorithm's steps. For the rest of this section, we will say that a point is ``feasible'' when it is in the feasible region of \eqref{eqn:main_problem}. Also, given a point $x$ we will call the projection of $\nabla f(x)$ to the nullspace of constraint normals of \eqref{eqn:working_set_problem}, $\begin{bmatrix} \bar A_k \\ x^T \end{bmatrix}$, as ``projected gradient of $f(\cdot)$ at $x$''.

\begin{algorithm}
	\label{alg:active_set}
	\caption{Active Set method for solving \eqref{eqn:main_problem} when $r_{\min} = r_{\max} \eqdef r$}
	Assume a feasible starting point $x_0$ \\
	Set $W_0$ to be a subset of the active constraints in $x_0$ \\
	\For{$k=\newcontent{0,} 1, 2, \dots$}{
		Solve \eqref{eqn:working_set_problem} with Algorithm \ref{alg:trs}\;
		\uIf{Algorithm \ref{alg:trs} returned two distinct minimizers}{
			$x_{k+1} \leftarrow$ solution of the 2D subproblem defined by $x_k$ and the two returned minimizers.
		}
		\Else{
			$x_{k+1} \leftarrow$ solution of the 2D subproblem defined by $x_k$, the returned (global) minimizer and the projected gradient of $f$ at $x_k$.
		}
		\uIf{$x_{k+1}$ is not a minimizer of \eqref{eqn:working_set_problem} and no new constraint was ``hit''}{
			Run projected gradient descent starting from $x_{k+1}$ until a new constraint is ``hit'' or until convergence\newcontent{, and store the resulting point in $x_{k+1}$}\;
			\uIf{PGD converged to a feasible \newcontent{$x_{k+1}$} of indefinite projected Hessian}{
					Compute a suitable limiting direction $d$ along which $x_{k+1}$ is a local maximum\;
					$x_{k+1} \leftarrow$ solution of the $2D$ subproblem on the plane defined by \newcontent{$x_{k + 1}$}, a global minimizer of \eqref{eqn:working_set_problem} and $d$\;
			}
		}
		\uIf{a new constraint was ``hit''} {
			Obtain $W_{k+1}$ by adding one of the blocking constraints to $W_k$\;
		}
		\Else{
			$W_{k+1} \leftarrow \mathrm{check\_multipliers}(x_{k+1}, W_k)$\;
		}
	}
	\BlankLine
	\SetKwFunction{Facq}{check\_multipliers}
	\Fn{\Facq{$x, W$}}{
		Compute the Lagrange multipliers $(\kappa, \mu)$ at $x$ that satisfy \eqref{eqn:KKT1}-\eqref{eqn:KKT4} with $A = [a_i]_{i \in \mathcal{W}}$\;
		\uIf{$\kappa \geq 0$}{
			\textbf{Terminate Algorithm} with \newcontent{$x$ as the returned solution}
			}\Else{
			$j \leftarrow \argmin_{j \in W} \kappa$\;
			$W \leftarrow W \setminus \{j\}$\;
		}
	}
\end{algorithm}

\subsubsection{Move towards the minimizers of \eqref{eqn:working_set_problem}}
 Notice that Algorithm \ref{alg:trs} always returns a global minimizer $x^g$ of \eqref{eqn:working_set_problem} and possibly a second minimizer $x^*$.

Let us consider first the case where two distinct minimizers are returned which is treated in \underline{lines 5-6} of Algorithm \ref{alg:active_set}. In this case it is always possible to either set $x_{k+1}$ to a minimizer of \eqref{eqn:working_set_problem} or ``hit'' a new constraint as follows. Consider the circle defined by the intersection of the sphere $\norm{x}_2 = r$ and the plane defined by $x_k$ and the two returned minimizers. When \eqref{eqn:working_set_problem} is constrained in this circle, then it can be reduced to a two-dimensional TRS. As such, we will show that it has at most 2 maximizers and 2 minimizers, except perhaps when $x_k$ is a global minimizer of \eqref{eqn:working_set_problem} in which case we can trivially set $x_{k+1} = x_{k}$.
\begin{proposition} \label{prop:2d_trs}
	A two dimensional TRS has at most 2 minimizers and 2 maximizers except when its objective is constant across its feasible region.
\end{proposition}
\begin{proof}
	It suffices to show the result only for the minimizers, as a negation of the TRS's objective would show the same result for its maximizers. Suppose that the TRS does not belong to the hard case. Then it has a unique global minimizer and at most one local-nonglobal minimizer \cite{Martinez1994}. If it belongs to the hard case, then only global minimizers can exist. These are intersections of the affine subspace \eqref{eqn:hard_case_subspace} and the sphere \eqref{eqn:hard_case_norm}. Note that this intersection is either a distinct point, two points or a circle. In the latter case, every feasible point of \eqref{eqn:trs} is a global minimizer. This completes the proof.
\qed\end{proof}
We can identify the two minimizers of this two-dimensional TRS as $x^g$ and $x^*$. It follows that at least in one of the circular arcs connecting $x_k$ with $x^g$ and $x^*$ the objective value is always less than $f(x_k)$. Thus, by moving into that circular arc one will either end up with $x^g$ or $x^*$ or ``hit'' a new constraint.

On the other hand, if a single global minimizer $x^g$ was returned, which corresponds to \underline{lines \newcontentB{7-8}} of Algorithm \ref{alg:active_set}, then we consider the circle defined by the intersection of the sphere $\norm{x}_2 = r$ and the plane defined by $x_k$, \newcontentB{$x^g$} and the projected gradient of $f(\cdot)$ at $x_k$. Obviously, if $x^g$ is feasible for \eqref{eqn:main_problem} then we choose $x_{k+1} = \newcontentB{x^g}$, but otherwise we are not guaranteed to ``hit'' a new constraint. This is because the associated two-dimensional TRS might possess a second minimizer, which is not necessarily a stationary point of \eqref{eqn:working_set_problem}, but might be the best feasible solution on this circle.

\subsubsection{Perform projected gradient descent (if necessary)}
This subsection concerns the case where the procedure described in the previous subsection could neither set $x_{k+1}$ to a minimizer of \ref{eqn:working_set_problem} nor ``hit'' a new constraint. In this case (\underline{lines 10-14} of Algorithm \ref{alg:active_set}), we proceed by performing projected gradient descent (PGD) on \eqref{eqn:working_set_problem} starting from the current iterate of the active-set algorithm. This is guaranteed to converge to a KKT point $x^s$ with $f(x^s) \leq f(x_k)$ \cite[Theorem 4.5]{Beck2018}. Suppose that $x^s$ is not feasible for \eqref{eqn:main_problem}. Since the feasible region is a closed set, PGD exits the feasible region in finitely many steps. By stopping the projected gradient descent just before it exits the feasible region, we can find a point $x_{k+1}$ with a newly ``hit'' constraint.

Consider now the case where $x^s$ is feasible for \eqref{eqn:main_problem}. If the minimum eigenvalue $\lambda_{\min}$ of the projected Hessian of $f(\cdot)$ at $x^s$ is nonnegative (\underline{lines 12-14} of Algorithm \ref{alg:active_set}), then we set $x_{k+1} = \newcontent{x^s}$ and we proceed to the steps outlined in the next subsection. Otherwise $\lambda_{\min} < 0$, and there exists a feasible sequence $\{z_k\}$ (w.r.t. to \eqref{eqn:working_set_problem}) converging to $x^s$ with an appropriate limiting direction $d$
\begin{equation}
	d \eqdef \lim_{k \to \infty} \frac{z_k - x^s}{\norm{z_k - x^s}_2},
\end{equation}
such that
\begin{equation} \label{eqn:taylor_direction}
	f(\newcontent{z_k}) \newcontentB{=} f(x^s) + \frac{1}{2}\lambda_{\min}\norm{z_k - x^s}^2 + \smallO\left(\norm{z_k - x^s}^2\right).
\end{equation}
In practice, we can choose $d$ equal to some projected eigenvector associated with $\lambda_{\min} < 0$. Consider now the circle defined by the intersection of the sphere $\norm{x}_2 = r$ and the plane defined by $x^s, x^g$ and the limiting direction~$d$. As before, this can be reduced to a two-dimensional TRS that possesses at most 2 minimizers and 2 maximizers (Proposition \ref{prop:2d_trs}). We can identify two of them: $x^g$ which is a global minimum and $x^s$ which, according to \eqref{eqn:taylor_direction}, is a local maximum. It follows that in at least one of the two circular arcs connecting $x^s$ with $x^g$ the objective value is always less than $f(x^s)$. Thus, by moving into that circular arc, and since $x^s$ is feasible for \eqref{eqn:main_problem} and $x^g$ infeasible, we can identify a suitable feasible point $x_{k+1}$ such that \newcontent{$f(x^g) \leq f(x_{k+1}) \leq f(x^s)$} for which a new constraint, not in the current working set, becomes active.

\subsubsection{Update the working set and check for termination}
If a new constraint was identified in the steps above, then we update the working set and proceed to the next iteration (\underline{lines 15-16} of Algorithm \ref{alg:active_set}). Otherwise, $x_{k+1}$ was set to a second-order necessary stationary point of \eqref{eqn:working_set_problem} and we proceed as follows (\underline{lines 17-18} of Algorithm \ref{alg:active_set}). Stationarity of $x_{k+1}$ gives
\[
	\nabla f(x_{k+1}) + \mu x_{k+1} + \sum_{i \in \mathcal{W}_k} a_i \kappa_i = 0
\]
for some Lagrange multipliers $\mu$ and $\kappa_i \; \forall i \in \mathcal{W}_k$. Thus, $(x_{k+1}, \mu, \kappa)$ satisfy the first KKT condition \eqref{eqn:KKT1} if we define $\kappa_i = 0 \; \forall \not \in \mathcal{W}_k$. It follows from feasibility of $x_{k+1}$ and the definition of $\kappa$ that the third and fourth KKT conditions \eqref{eqn:KKT3}-\eqref{eqn:KKT4} also hold. If we also have that $\kappa_i \geq 0 \; \forall i \in \mathcal{W}_k$, then $(x_{k+1}, \mu, \kappa)$ is a KKT pair for \eqref{eqn:main_problem} and we terminate the algorithm (\underline{lines 23-24} of Algorithm \ref{alg:active_set}). If, on the other hand, $\kappa_i < 0$ for some $i \in \mathcal{W}_k$ (\underline{lines 25-27} of Algorithm \ref{alg:active_set}), then \eqref{eqn:KKT2} is not satisfied and $x_{k+1}$ is not a local minimizer for \eqref{eqn:main_problem}. In fact, the objective $f(\cdot)$ may be decreased by dropping one of the constraints corresponding to a negative Lagrange multiplier \cite[Section 12.3]{Nocedal2006}.

\subsubsection{Finite termination}
We will show that Algorithm \ref{alg:active_set} terminates in finitely many outer iterations, assuming that $x_{k+1} \neq x_k$ for every $k = 1, 2. \dots$. Indeed at every iteration the algorithm either
\begin{itemize}
  \item Moves to a stationary point of the current TRS subproblem; or
  \item Activates a new constraint.
\end{itemize}
Since there can be at most $m$ constraints in the working set it follows that $x_k$ visits a stationary point of the $k-$th TRS subproblem periodically (at least every m iterations). Furthermore, note that every TRS subproblem, resulting from a particular working set, has at most $2n$ sets of stationary points of the same objective value \cite[Theorem 4.1 and subsequent comments]{Forsythe1965}. Since there are a finite number of different working set configurations, it follows that there exist finitely many such sets. On the other hand, every constraint that is deleted from the working set has an associated negative Lagrange multiplier, thus every iterate $x_k$ is not a stationary point of the next iteration. Furthermore, note that if $x_k$ is not a stationary point for the $k+1$ subproblem then our algorithm generates $x_{k+1}$ with $f(x_{k+1}) < f(x_{k})$ unless $x_k \newcontent{=} x_{k+1}$ which is excluded by assumption. This means that once the algorithm visits one of these sets of stationary points with equal objective value, it can never visit it again. Hence, the algorithm terminates after finitely many iterations.

We believe that it is not possible to bound the number of iterations by a polynomial in $m, n$. Indeed, even finding an initial feasible point is NP-complete, as we show in \S\ref{subsec:initial_point}.

\subsubsection{Further Remarks}
Similarly to active-set algorithms for quadratic programs, we can always update the working set such that $\bar{A}_k$ is full row rank. However, LICQ might still fail to be satisfied when the gradient of the spherical constraint lies on $\mathcal{R}(\bar{A}^T_k)$. In these cases we terminate the algorithm without guarantees about local optimality.
\subsection{Solving \eqref{eqn:main_problem} for any $r_{\min}, r_{\max}$} \label{sec:main_problem}
We are now ready to present an active-set algorithm that solves \eqref{eqn:main_problem} for any $r_{\min}$ and $r_{\max}$. At every iteration of the suggested active-set algorithm, the spherical inequality constraint will either be in the working set or not. If it is, then we iterate as described in Algorithm \ref{alg:active_set}; otherwise we proceed with a generic (nonconvex) quadratic programming active-set algorithm \cite[\S 16.8]{Nocedal2006}\newcontent{, \cite{Gill1991}, \cite{Gould2002}}. We switch between the two algorithms when an iterate of the QP algorithm hits the spherical boundary or when the Lagrange Multiplier $\mu$ of the norm constraint in Algorithm \ref{alg:active_set} is negative and less than any other Lagrange Multiplier $\kappa_i$.

In the special the case where $r_{\min} = 0$, the norm constraint reduces to $\norm{x}_2 \leq r_{\max}$. In this case, when a switch from Algorithm \ref{alg:active_set} to the QP Algorithm happens, the projected Hessian of the Lagrangian of $f(\cdot)$ at $x_k$ is positive semidefinite, i.e.
\begin{equation}
	Q^T (P + \mu I)Q \succeq 0
\end{equation}
where $Q \eqdef [Q_1 \; q_2]$ is defined by the ``thin'' QR decomposition of $[\bar A_k^T \; x_k]$ and $q_2$ is an appropriate vector. Recall that for a switch to happen, we need $\mu < 0$; hence $Q^TPQ$ is positive definite. Thus, $Q_1^TPQ_1$, i.e.\ the projected Hessian of the next iteration that is to be handled by the generic QP algorithm, has at most one nonpositive eigenvalue \cite[Corollary 4.1]{Saad2011}. This means that even the popular, but less flexible, class of ``inertia controlling'' QP algorithms can be used as part of the suggested active-set algorithm for solving \eqref{eqn:main_problem} when $r_{\min} = 0$.
\subsubsection{Finding an initial point} \label{subsec:initial_point}
Algorithm \ref{alg:active_set} is a primal active-set algorithm and, as such, it requires an initial feasible point. Unfortunately, finding such a point is, in general, NP-complete as we formally establish at the end of this subsection. Nevertheless, we can find a feasible point for \eqref{eqn:main_problem} with standard tools, albeit in exponential worst-case complexity, as we proceed to show. Indeed, define the following problems:
\begin{multicols}{2}\noindent
  \begin{equation*}
		\begin{array}{ll}
			\text{maximize}   & x^T x\\
			\text{subject to} & Ax \leq b \\
											  & \norm{x}_{\infty} \leq r_{\min}
		\end{array}
	\end{equation*}\noindent
	\begin{equation*}
		\begin{array}{ll}
			\text{minimize}   & x^T x\\
			\text{subject to} & Ax \leq b
		\end{array}
	\end{equation*}
\end{multicols}
\noindent i.e.\ a convex minimization and a convex maximization problem. Denote with $x^*_{\min}$ and $x^*_{\max}$ the solutions to the convex minimization and maximization problems respectively. Then, \eqref{eqn:main_problem} is feasible iff $\norm{x^*_{\min}}_2 \leq r_{\max}$ and $\norm{x^*_{\max}}_2 \geq r_{\min}$, in which case a feasible point for \eqref{eqn:main_problem} can be found by interpolating between $x^*_{\min}$ and $x^*_{\max}$.

The convex minimization problem above can be solved in polynomial time with e.g.\ interior point or first order \cite{Stellato2017} methods. On the other hand, the convex maximization problem can have exponential worst-case complexity, but it can be solved to local or even global optimality with standard commercial solvers e.g.\ with \texttt{IBM CPLEX}.

Finally, we formally show that determining if \eqref{eqn:main_problem} is feasible is NP-complete:
\begin{proposition} \label{prop:np_feasibility}
	Determining if there is a solution to
	\begin{equation}\label{eqn:feasibility} \tag{F}
		\begin{array}{ll}
			\text{find}  & x \\
			\text{subject to} & Ax \leq b \\
				& r_{\min} \leq \norm{x}_2 \leq r_{\max} \\
		\end{array}
	\end{equation}
	is NP-complete.
\end{proposition}
\begin{proof}
	Determining if \eqref{eqn:feasibility} has a solution is NP since we can easily check whether a candidate point $x$ is feasible for \eqref{eqn:feasibility}. Furthermore, it can be decomposed into the following two independent problems: (i) ``is there a point in the polytope $Ax \leq b$ such that $\norm{x}_2 \leq r_{\max}?$'' and (ii) ``is there a point in the polytope $Ax \leq b$ such that $\norm{x}_2 \geq r_{\min}?$'' Problem (i) can be answered in polynomial time with e.g. interior point methods. Thus determining if \eqref{eqn:feasibility} has a solution is reducible to problem (ii) which is known to be NP-complete \cite[Problem 3]{Murty1987}.
\qed\end{proof}

%% file: sections/results_first.tex
\section{Applications and Experiments} \label{sec:experiments}
In this section we present numerical results of the suggested active-set algorithms on a range of numerical optimization problems. A \texttt{Julia} implementation of the algorithm, along with code for the generation of the results of this section is available at:\\
\centerline{\texttt{https://github.com/oxfordcontrol/QPnorm.jl}}\\
As described in the previous sections, the algorithm is based on a TRS solver, and a general (nonconvex) Quadratic Programming solver. We implemented these dependencies as separate packages that we also release publicly as listed below.
\begin{itemize}
	\item \texttt{TRS.jl}: A \texttt{Julia} package for the computation of global and local-nonglobal minimizers of
	\begin{equation}\label{eqn:main_problem} \tag{P}
		\begin{array}{ll}
			\text{minimize}   & \frac{1}{2}x^T P x + q^T x \\
			\text{subject to} & \norm{x} = r \quad \text{or} \quad \norm{x} \leq r \\
												& Ax = b,
		\end{array}
	\end{equation}
	essentially implementing in \texttt{Julia} the theoretical results of Section \ref{sec:trs} and \cite{Adachi2017}. Available at:\\
	\centerline{\texttt{https://github.com/oxfordcontrol/TRS.jl}}
	\item \texttt{GeneralQP.jl}: A \texttt{Julia} implementation of \cite{Gill1978}, i.e.\ an ``inertia controlling'' active-set solver for nonconvex, dense quadratic problems, with efficient and numerically stable routines for updating QR decompositions of the working set and LDLt factorizations of the projected Hessian. Available at:\\
	\centerline{\texttt{https://github.com/oxfordcontrol/GeneralQP.jl}}\\ This solver is useful as a part of the suggested algorithm for solving \eqref{eqn:main_problem} according to the remarks of \S\ref{sec:main_problem}.
\end{itemize}
For simplicity our implementations of the active-set algorithms are based on dense linear algebra that uses a QR factorization to compute/update a nullspace basis for every working set. Thus the presented results are limited to dense problems, except for \S\ref{sec:pca} where the special structure of the constraints of \eqref{eqn:scotlass_reformulated} results in trivial, sparse QR factorizations of the nullspace bases.

Before presenting the results we will discuss some practical considerations regarding the suggested active-set algorithm. The eigenproblems \eqref{eqn:eigenvalue_problem} associated with the trust-region subproblems of each working set are solved with \texttt{ARPACK} \cite{Lehoucq1998}. In the occasional cases where \texttt{ARPACK} fails, a direct eigensolver is used. Finally, before solving the subproblem \eqref{eqn:working_set_problem} of every iteration, we perform a few projected gradient steps with the hope to quickly activate a new constraint.

We now proceed with the results, starting with the simplest case of dense random problems.

\subsection{Random Dense Constant Norm QPs}
We compared the performance of our algorithm against the state-of-the-art non-linear solver \texttt{Ipopt} \cite{Wachter2006} with its default parameters. We use the \texttt{Julia} interface of \texttt{Ipopt}, \texttt{Ipopt.jl}, which exhibits negligible interface overhead times due to the excellent interfacing capabilities of \texttt{Julia} with \texttt{C++}.

We consider a set of these randomly generated problems with varying number of variables $n$. A feasible point for each problem instance is calculated for our Algorithm as described in Section \ref{subsec:initial_point}. The time required to compute the initial feasible points is included in the subsequent results.

\newcontent{Figure \ref{fig:random} (left)} shows the execution times. We observe that our algorithm is significantly faster than \texttt{Ipopt} by a factor of up to 50. Both of the solvers achieve practically identical objectives (relative difference less than $10^{-9}$) in all of the problem instances, except in two cases where there is a considerable difference due to the fact that the solvers ended up converging in different local minimizers. Finally, Figure \ref{fig:random} (right) shows the infeasibility of the returned solution, where we observe that our solver returns solutions of significantly smaller (i.e.\ better) infeasibility.
\begin{figure}
	\caption{Timing results (left) and maximum feasibility violation (right) of the solution of random problems with varying size $n$ generated with the following parameters $P =$ Symmetric(randn$(n, n)$), $q =$ randn($n$), $A =$ randn($1.5n, n$), $b =$ randn($n$) and $r = 100$. Hardware used: Intel E5-2640v3, 64GB memory. The maximum feasibility violation is defined as $\max((Ax^* - b)_1, \dots, (Ax^* - b)_m, \abs{\norm{x^*}_2^2 - r^2}, 0)$ for a solution $x^*$.}
	\label{fig:random}
	\centering
	\begin{minipage}{.49\textwidth}
	\includegraphics[width=.99\textwidth]{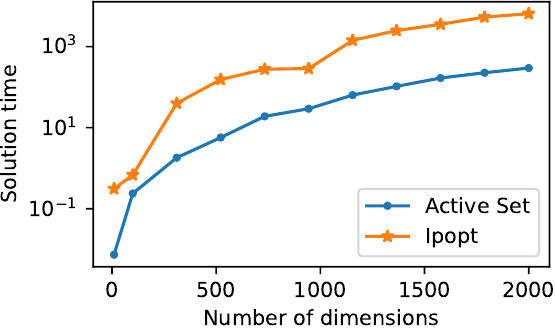}
	\end{minipage}
	\centering
	\begin{minipage}{.49\textwidth}
	\includegraphics[width=.99\textwidth]{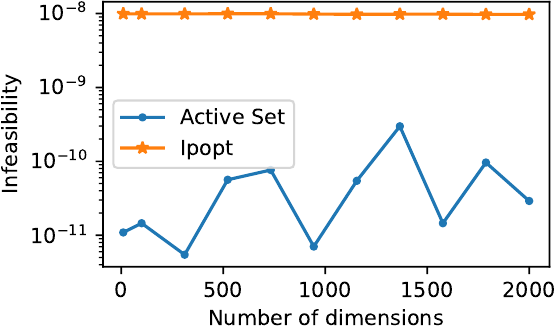}
	\end{minipage}
\end{figure}

%% file: sections/results_cutest.tex
\subsection{Computing search directions for Sequential Quadratic Programming} \label{sec:sqp}
Sequential Quadratic Programming (SQP) is a powerful and popular algorithm that aims to solve the problem
\begin{equation} \label{eqn:nlp}
	\begin{array}{ll}
		\text{minimize}   & g(x) \\
		\text{subject to} & h(x) \leq 0,
	\end{array}
\end{equation}
where $g : \mathbb{R}^n \mapsto \mathbb{R}$ and $h : \mathbb{R}^n \mapsto \mathbb{R}^m$ are general (nonlinear) functions. At every iteration, SQP minimizes the nonlinear objective over some \emph{search directions}. These search directions can, for example, be obtained by the optimization of a quadratic approximation of the original objective subject to some linear approximation of the original constraints. It is common to introduce a norm constraint $\norm{\delta x}_2 \leq r$ that ensures that the solution will be inside a confidence or \emph{trust} region of this approximation thus resulting in a problem of the form \eqref{eqn:main_problem}. This is particularly useful at points where the Hessian of the nonlinear objective is singular or indefinite as the search directions computed by the respective QPs might be unrealistically large or unbounded.

Typically, SQP algorithms solve problems simpler than \eqref{eqn:main_problem} by e.g.\ not including the inequality constraints \eqref{eqn:main_problem} in explicit form for the calculation of the search directions \cite[\S 18.2]{Nocedal2006}. We will show, however, that our solver is capable of computing search directions from \eqref{eqn:main_problem} directly, without the need for these simplifications.

We demonstrate this on all the problems of the \texttt{CUTEst} collection \cite{Gould2015} that have linear constraints and number of variables less than $2000$. For each of these problems we consider the quadratic approximation
\begin{equation}
	\label{eqn:sqp}
	\begin{array}{ll}
		\text{minimize}   & f(\delta x) \eqdef \frac{1}{2} \delta x^T \nabla^2g(x_0) \delta x + \delta x^T \nabla g(x_0) + c \\
		\text{subject to} & \norm{\delta x}_2^2 \leq 1 \\
											& \nabla h(x_0) \delta x   + h(x_0) \leq 0
	\end{array}
\end{equation}
where $x_0$ is the closest feasible point (in the 2-norm) to the initial point suggested by \texttt{CUTEst}, which we compute with \texttt{GUROBI} \cite{gurobi}\newcontent{, and $\delta x \eqdef x - x_0$}. We discard problems where \texttt{GUROBI} fails to calculate an initial feasible point. Furthermore, we do not consider problems for which \eqref{eqn:sqp} is convex, since these problems can be solved to global optimality in polynomial time with standard solvers such as \texttt{MOSEK} or \texttt{COSMO.jl} \cite{Garstka2019}. Finally, some problems in the \texttt{CUTEst} collection are parametric; if the default parameters lead to a problem with more than 2000 variables then we choose a parameter that, if possible, leads to a number of variables close to, but not exceeding, 2000. Tables \ref{tab:cutest}-\ref{tab:cutest_3} list all the 63 problems considered, along with any non-default parameters.

We compare the performance of our algorithm against \texttt{Ipopt} with its default options. For some problems, \texttt{Ipopt} terminates without indication of failure but returns a low quality solution. We consider a problem as ``solved'' using the same criteria as \cite{Ferreau2014}. That is, we require that the overall error of the KKT conditions is less than $10^{-4}$, defined as
\begin{align}
	\label{eqn:error}
	\text{error} &\eqdef \max(\epsilon_{\text{primal inf}}, \epsilon_{\text{dual inf}}, \epsilon_{\text{stationarity}}, \epsilon_{\text{compl}})
	\intertext{with}
	\nonumber
	\epsilon_{\text{primal inf}} &\eqdef \max\left(0, \delta x^T \nabla h(x_0) + h(x_0), \norm{\delta x}_2 - 1\right) \\
	\nonumber
	\epsilon_{\text{dual inf}} &\eqdef -\min(0, \kappa_1, \dots, \kappa_m, \mu) \\
	\nonumber
	\epsilon_{\text{stationarity}} &\eqdef \norm{\nabla^2g(x_0) \delta x + \nabla g(x_0) + \kappa^T \nabla h(x_0) + 2 \mu \newcontent{\delta x}}_{\infty} \\
	\nonumber
	\epsilon_{\text{compl}} &\eqdef
		\max(\{\min(\kappa_i, \newcontent{\|\nabla h(x_0) \delta x + h(x_0) \|_i)}\}, \min(\mu, \abs{\norm{\newcontent{\delta x}}_2^2 - 1})),
\end{align}
and $\kappa, \mu$ are the Lagrange multipliers corresponding to the linear inequality and norm constraints of \eqref{eqn:sqp}. Figure \ref{fig:performance_profiles} shows the performance profile for the problems considered. The performance profile suggests that our algorithm significantly outperforms \texttt{Ipopt} on this set of problems especially in reliability, in the sense that we consistently obtain high quality solutions. Note that unlike the dense implementation of our solver, \texttt{Ipopt} can exploit the sparsity of the problems efficiently. Further speedups can be brought to our algorithm with a sparse implementation which might allow its use in large scale sparse problems. Moreover, as an active-set solver, our algorithm can efficiently exploit warm starting, which might be highly desirable in SQP where repeated solution of problems of the form \eqref{eqn:sqp} is required as part of the SQP procedure for minimizing \eqref{eqn:nlp}. This is in contrast to \texttt{Ipopt} whose Interior Point nature makes warm starting very difficult to exploit.
\begin{figure}
	\centering
	\includegraphics[width=.6\textwidth]{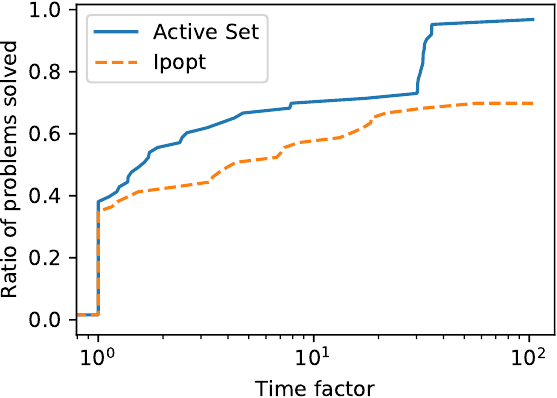}
	\caption{Performance graph of the suggested algorithm against \texttt{Ipopt} on calculating SQP search directions on all 63 problems of the \texttt{CUTEst} library listed in Tables \ref{tab:cutest}-\ref{tab:cutest_3}. Matrices are passed to \texttt{Ipopt} in a sparse format. Hardware used: Intel E5-2640v3, 64GB memory.}
	\label{fig:performance_profiles}
\end{figure}

%% file: sections/results_pca.tex
\subsection{Sparse Principal Component Analysis} \label{sec:pca}
Principal Component Analysis (PCA) is a standard, widely used dimensionality reduction algorithm. It has applications in numerous fields, including statistics, machine learning, bioinformatics, genetics, meteorology and others. Given a $k \times n$ data matrix $D$ that consists of $k$ points of $n$ variables, PCA suggests a few linear combinations of these variables, which we call \emph{principal vectors}, that explain as much variance of the data as possible. Standard PCA amounts to the following problem
\begin{equation*}
	\begin{array}{ll}
		\text{maximize}   & x^T \Sigma x \\
		\text{subject to} & \norm{x}_2 \leq 1
	\end{array}
\end{equation*}
where $\Sigma$ is the covariance matrix of the data. We will assume for the rest of this section that the data matrix $D$ is ``centered'', i.e.\ that it has column-wise zero mean. We can then consider $\Sigma$ to be the empirical covariance matrix $\Sigma \eqdef \frac{D^T D}{k-1} \in \mathbb{S}_+^n$. The above problem is essentially an eigenvalue problem on $\Sigma$ or a singular value problem on $D$ and, as such, can be solved with standard linear algebra tools.

In general, each principal vector is a linear combination of all the variables, i.e.\ typically all components of the principal vectors are non-zero. This can pose an issue of interpretability of the reduced dimensions, as it is often desirable to express the principal vector as a combination of a few variables, especially when the variables are associated with a user interpretable meaning. To alleviate this problem, sparsity has to be enforced in the original PCA problem, resulting in a new optimization problem that aims to identify a small set of variables the linear combination of which will hopefully still be able to explain a significant proportion of the variance of the data.

Unfortunately, enforcing sparsity in PCA results in a combinatorial optimization problem. Various remarkably efficient and scalable heuristics have been suggested to avoid this intractability \cite{Journee2010}, \cite{Yuan2013}. We will focus on one of the most popular heuristics that is based on a lasso type constraint originally introduced by \cite{Jolliffe2003}, resulting in the following optimization problem:
\begin{equation} \label{eqn:scotlass}
	\begin{array}{ll}
		\text{minimize}   & f(x) \eqdef -x^T \Sigma x \\
		\text{subject to} & \norm{x}_2 \leq 1                         \\
											& \norm{x}_1 \leq \gamma,
	\end{array}
\end{equation}
where $x \in \mathbb{R}^n$ is the decision variable and $\gamma \geq 1$ is a parameter that controls the sparsity of the solution.

We will now show how Problem \eqref{eqn:scotlass} can be solved with our algorithm. Note that \eqref{eqn:scotlass} is not in the standard norm-constrained form \eqref{eqn:main_problem} that we address. However, we show in the Appendix that it is equivalent to
\begin{equation} \label{eqn:scotlass_reformulated}
	\begin{array}{ll}
		\text{minimize} & g(w) \eqdef -(w_1 - w_2)^T
			\Sigma(w_1 - w_2) \\
		\text{subject to} & \norm{w}_2^2 \leq 1 \\
											& \mathbf{1}^T w \leq \gamma \\
											& w \geq 0,
	\end{array}
\end{equation}
where $[w_1^T\; w_2^T] \eqdef w^T \in \mathbb{R}^{2n}$,
in the sense that \eqref{eqn:scotlass} and \eqref{eqn:scotlass_reformulated} have the same optimal value and every minimizer of \eqref{eqn:scotlass_reformulated} $\bar w$ defines a minimizer $\bar x = \bar w_1 - \bar w_2$ for \eqref{eqn:scotlass}. Also, the optimal $\bar w$ has the same cardinality as the respective optimal $\bar x = \bar w_1 - \bar w_2$. 

The reader might think that solving \eqref{eqn:scotlass_reformulated} with an active-set method \newcontent{requires} a solve time that is polynomial in the number of variables, thus limiting its scalability as compared to simpler gradient based methods that are scalable but might have inferior accuracy and, potentially, weaker convergence guarantees. This is not true however, since our algorithm takes advantage of the sparsity of the iterates. Indeed, excluding the lasso constraint, the working set $\mathcal{W}_k$ of Algorithm \ref{alg:active_set} applied in Problem \eqref{eqn:scotlass_reformulated} can be interpreted as the set of variables $w_i$ that are fixed to zero. Thus every subproblem can be trivially reduced to $k$ dimensions, where $k$ is the number of variables $w_i$ that are not included in the current working set. This property is particularly attractive when the user is interested in a very sparse solution and has been recognized in the literature on $\ell_1$-penalized  active-set algorithms \cite[\S 6]{Bach2012}.

We will demonstrate the scalability and flexibility of our algorithm by applying it in a large ``bag of words'' dataset from articles of the  New York Times newspaper obtained by the University of California, Irvine (UCI) Machine Learning Repository\footnote{Available at \texttt{http://archive.ics.uci.edu/ml/datasets/Bag+of+Words}}. In this dataset, every row of the associated data matrix $D$ corresponds to a word used in any of the documents and every column to a document. The entries of the matrix are the number of occurrences of a word in a document. The dataset contains $m \approx 300,000$ articles that consist of $n \approx 102,660$ words (or keywords) with $\approx 70$ million non-zeros in the dataset matrix $D$. We seek to calculate a few sparse principal vectors, each of which will hopefully correspond to a user-interpretable category, like politics, economics, the arts etc.

The usefulness of Sparse PCA on this dataset has been already demonstrated in the literature. Both \cite{Richtarik2012} and \cite{Zhang2011} have generated a set of five sparse principal vectors, each of which is a linear combination of five words. The results of \cite{Richtarik2012} are listed in Table \ref{tab:pca_richtarik}. The five resulting  principal vectors have distinct, user interpretable associated meanings that correspond to sports, economics, politics, education, and US foreign policy. The resulting words can be used to create user interpretable categories of the documents \cite{Zhang2011}. However, this interpretability is in general lost when we increase the cardinality of the sparse principal vectors, i.e.\ when we require more representative words for each category. For example, the result of applying \texttt{TPower} \cite{Yuan2013} (a state-of-the art algorithm for Sparse PCA) in the NYtimes dataset to generate a sparse principal vector that consists of 30 words, listed in Table \ref{tab:pca_tpower}, does not have a distinct associated meaning. For reference, \texttt{TPower} runs in $4.91$ seconds on a standard Laptop with an Intel i7-5557U CPU and 8GB of memory and has a resulting variance of $58.79$. Another popular algorithm, $\ell_1-$\texttt{GPower}, runs in $355.86$ seconds and results in a variance of $56.66$. The results of our  active-set algorithm are better than $\ell_1-$\texttt{GPower}\footnote{Note that $\ell_1$ penalized algorithms (such as our algorithm and $\ell_1-$\texttt{GPower}) do not allow for the specification of the desired sparsity directly. Instead, their sparsity is controlled by a penalty parameter. We find a penalty parameter that results in 30 nonzeros via means of a binary search, while exploiting warm starting of the solvers.} and comparable (slightly inferior) to \texttt{TPower}: our algorithm runs in $5.22$ seconds and has a resulting variance equal to $58.23$. However, our algorithm is more general, in the sense that it can solve any problem of the form \eqref{eqn:main_problem} which, as we proceed to show, allows for variants of \eqref{eqn:scotlass} that can generate principal vectors with more meaningful interpretation.

\begin{table}
	\caption{Results of \cite{Richtarik2012} on applying sparse PCA on the NYtimes dataset. Each column lists words that correspond to nonzero entries in the respective principal vectors. The resulting  principal vectors have clear associated meanings that correspond to sports, economics, elections, education, and the US.}
	\label{tab:pca_richtarik}
	\centering
	\begin{tabular}{lllll}
		\toprule
	Vector \#1  & Vector \#2  & Vector \#3         & Vector \#4  & Vector \#5 \\
	\midrule
	game        & companies    & campaign          & children    & attack     \\
	play        & company      & president         & program     & government    \\
	player      & million      & zzz\_al\_gore     & school      & official   \\
	season      & percent      & zzz\_bush         & student     & zzz\_u\_s    \\
	team        & stock        & zzz\_george\_bush & teacher     & zzz\_united\_states
	\end{tabular}
\end{table}

Indeed, improving interpretability of the results of Sparse PCA can be achieved by the incorporation of additional constraints in the underlying optimization problem. Unfortunately, most of the algorithms that are specific for sparse PCA do not allow for the presence of extra constraints. An exception is perhaps \texttt{GPower} of \cite{Journee2010}, which is a generic framework for minimizing a concave function $f$ over a compact subset $\mathcal{F}$. The iterates of \texttt{GPower} are generated as
\begin{equation} \label{eqn:gpower}
	\begin{array}{rll}
		x_{k+1} = &\text{argmin}   & x^T f'(x_k) \\
			&\text{subject to} & x \in \mathcal{F}
		\end{array}
\end{equation}
where $f'(\cdot)$ is a subgradient of $f(\cdot)$. \cite{Journee2010} showed that for the case of Sparse PCA of \eqref{eqn:scotlass} the above optimization problem reduces to two matrix multiplications with $D$ and $D^T$. However, in the presence of general linear inequality constraints, solving \eqref{eqn:gpower} can be at least as hard as solving a linear program (LP), which is obviously considerably more computationally demanding than two matrix multiplications.

In contrast, due to the generality of our approach, imposing any linear constraint is straightforward for our algorithm. To demonstrate this, we impose a non-negativity constraint on the elements of the principal vector to avoid cases where the principal vectors consist of elements that have opposing meanings which happens, for example, in the results of \texttt{TPower} listed in Table \ref{tab:pca_tpower}. Table \ref{tab:nonnegative_pca} lists five sparse principal vectors obtained with this approach which take 24.37 seconds to compute on a standard Laptop with an Intel i7-5557U CPU and 8GB of memory. Note that, unlike the results of Table \ref{tab:pca_tpower}, the principal vectors of Table \ref{tab:nonnegative_pca} have a clear associated meaning that can be labeled as sports, economics, politics, US foreign policy and education. These words could for example be used for organizing the documents in a user interpretable way \cite{Zhang2011}.

%% file: sections/tables_pca.tex
\begin{table}
	\sisetup{round-mode=places}
	\caption{Results of applying sparse PCA on the NYtimes dataset with \texttt{TPower}, allowing weights with mixed signs. Each column lists words that correspond to nonzero entries in the generated principal vector. Note that increasing the cardinality of the sparse principal vector has severely affected the interpretability of the principal vector as compared to the results of Table \ref{tab:pca_richtarik} that have clear associated meaning.}
	\label{tab:pca_tpower}
	\centering
	\begin{tabular}{
		l
	*{1}{S[round-precision=2]}
	l
	*{1}{S[round-precision=2]}
	}
	\toprule
	{Words} & {Weights} & {Words (continued)} & {Weights (continued)} \\
	\midrule
	teamed                    & -0.45197 & hitch     & -0.083649 \\
	gameday                   & -0.3698 & playland                  & -0.081834\\
	seasonable                & -0.34698 & 	leaguer                   & -0.07607 \\
	playful                   & -0.28105 & 	nightcap                  & -0.072675 \\
	playa                     & -0.24374 & 	percentage                & 0.27263 \\
	gamesmanship              & -0.20957 & 	companywide               & 0.23443 \\
	coachable                 & -0.13991 & 	companion                 & 0.14581 \\
	pointe                    & -0.13315 & 	stockade                  & 0.12727 \\
	runaround                 & -0.11266 &	marketability             & 0.12368 \\
	win95                     & -0.1104 &	millionaire               & 0.11005 \\
	yardage                   & -0.10503 &	governmental              & 0.092267 \\
	guzzle                    & -0.093052 & 	billionaire               & 0.090433 \\
	player                    & -0.089249 & 	businesses                & 0.076009 \\
	wonder                    & -0.086247 & 	zzz\_bush\_administration & 0.072731 \\
	ballad                    & -0.085514 & 	fundacion                 & 0.070655 \\
	\end{tabular}
\end{table}

\begin{table}
	\caption{Results of applying non-negative sparse PCA on the NYtimes dataset with the suggested active-set algorithm. Each column lists words that correspond to nonzero entries in the respective principal vectors. The words are sorted in decreasing magnitude of the respective weights. Given a set of computed principal vectors $\{x_i\}$, the next principal vector is computed on the deflated data matrix $D - \sum (Dx_i) x_i^T$. The algorithm is initialized with the $30$ most positive, or most negative, entries (depending on which one leads to higher explained variance) of the first singular vector of the respective matrix. Furthermore we perform a standard post-processing polishing step at each resulting principal vector, as described in \cite[\S 4.2]{Journee2010}. Note that, unlike Table \ref{tab:pca_tpower} our algorithm results in principal vectors that have clear, user interpretable associated meanings that corresponds to sports, economics, politics, education, and US foreign policy. These words could be used for document classification \cite{Zhang2011}.}
	\label{tab:nonnegative_pca}
	\centering
	\begin{tabular}{lllll}
		\toprule
	Vector \#1  & Vector \#2 & Vector \#3        & Vector \#4          & Vector \#5 \\
	\midrule
	team        & percent    & zzz\_al\_gore     & official            & school     \\
	game        & company    & zzz\_bush         & government          & student    \\
	season      & million    & zzz\_george\_bush & attack              & children   \\
	player      & companies  & campaign          & zzz\_u\_s           & program    \\
	play        & stock      & president         & zzz\_united\_states & teacher    \\
	games       & market     & election          & zzz\_bush           & parent     \\
	point       & billion    & political         & military            & high       \\
	coach       & business   & zzz\_white\_house & zzz\_american       & percent    \\
	run         & fund       & republican        & palestinian         & public     \\
	win         & analyst    & voter             & zzz\_taliban        & education  \\
	yard        & money      & vote              & zzz\_afghanistan    & kid        \\
	guy         & zzz\_enron & presidential      & terrorist           & district   \\
	won         & firm       & democratic        & war                 & test       \\
	played      & investor   & tax               & leader              & family     \\
	hit         & sales      & zzz\_republican   & zzz\_israel         & college    \\
	ball        & industry   & administration    & group               & home       \\
	playing     & plan       & zzz\_clinton      & bin                 & child      \\
	league      & cost       & ballot            & laden               & class      \\
	left        & investment & votes             & country             & group      \\
	home        & quarter    & plan              & administration      & job        \\
	fan         & deal       & democrat          & security            & boy        \\
	shot        & financial  & zzz\_washington   & forces              & money      \\
	field       & customer   & zzz\_congress     & zzz\_israeli        & help       \\
	playoff     & economy    & poll              & zzz\_pakistan       & mother     \\
	night       & chief      & support           & american            & system     \\
	goal        & price      & candidate         & zzz\_washington     & friend     \\
	final       & executive  & zzz\_florida      & nation              & girl       \\
	start       & growth     & candidates        & terrorism           & private    \\
	quarterback & earning    & governor          & weapon              & voucher    \\
	football    & share      & vice              & foreign             & grade
	\end{tabular}
	\end{table}

%% file: sections/appendix_pca.tex
In this Section we show that the suggested active-set algorithm of \S\ref{sec:main_problem} can perform $\ell_1$ sparse PCA, that is, to solve the problem
\begin{equation*}
	\begin{array}{ll}
		\text{minimize}   & f(x) \eqdef -x^T \Sigma x \\
		\text{subject to} & \norm{x}_2 \leq 1                         \\
											& \norm{x}_1 \leq \gamma.
	\end{array}
\end{equation*}
\begin{theorem}
	Problems \eqref{eqn:scotlass} and \eqref{eqn:scotlass_reformulated} have the same optimal value and every (local) minimizer $\bar w$ of \eqref{eqn:scotlass_reformulated} defines a (local) minimizer $\bar x = \bar w_1 - \bar w_2$ for \eqref{eqn:scotlass}.
\end{theorem}
\begin{proof}
	First note that both Problems \eqref{eqn:scotlass} and \eqref{eqn:scotlass_reformulated} have a minimizer as their objective is smooth and their feasible set compact. For any feasible $x$ of \eqref{eqn:scotlass}, the choice
	\begin{equation*}
		w_1 = x_+ \quad \text{ and } \quad w_2 = -x_-,
	\end{equation*}
	where $x_+$ and $x_-$ are defined element-wise as follows
	\begin{equation*}
		x_+ \eqdef \max(x, 0) \quad \text{ and } x_- \eqdef \min(x, 0),
	\end{equation*}
	is feasible for \eqref{eqn:scotlass_reformulated} since $\norm{w_1}_2^2 + \norm{w_2}_2^2 = \norm{x}_2^2 = 1$ and $\mathbf{1}^T w_1 + \mathbf{1}^T w_2 = \norm{x}_1 \leq \gamma$, while $f(x) =  g([w_1;\; w_2])$ since
	\begin{align*}
		-(w_1 - w_2)^T \Sigma (w_1 - w_2) = -x^T \Sigma x.
	\end{align*}
	Thus the optimal value of \eqref{eqn:scotlass_reformulated} is less than that of \eqref{eqn:scotlass}.

	Moreover, for any minimizer $(\bar w_1, \bar w_2)$ of \eqref{eqn:scotlass_reformulated}, choosing $\bar x = \bar w_1 - \bar w_2$ gives $f(\bar x) = g([\bar w_1;\; \bar w_2])$ with $\bar x$ feasible for \eqref{eqn:scotlass} as we  proceed to show. Indeed, note that $\bar w_1^T \bar w_2 \geq 0$ and $\norm{\bar w_1}_2^2 + \norm{\bar w_2}_2^2 = 1$ thus
	\begin{equation}
		\norm{\bar x}_2^2 = \norm{\bar w_1}_2^2 + \norm{\bar w_2}_2^2 - 2\bar w_1^T \bar w_2 \leq \norm{\bar w_1}_2^2 + \norm{\bar w_2}_2^2 = 1
	\end{equation}
	and 
	\begin{equation}
		\norm{\bar x}_1 = \norm{\bar w_1 - \bar w_2}_1 \leq \norm{\bar w_1}_1 + \norm{\bar w_2}_1 \leq \gamma.
	\end{equation}

	Assume that $\bar x$ is not a minimizer for \eqref{eqn:scotlass}. Then, there exists an $\tilde x$ with $\norm{\tilde x - \bar x}_\infty$ arbitrarily small such that $f(\bar x) > f(\tilde x)$. Thus $(\tilde w_1, \tilde w_2) \eqdef (\tilde x_+, -\tilde x_-)$ is feasible for \eqref{eqn:scotlass_reformulated} with $ g([\tilde w_1;\; \tilde w_2]) = f(\tilde x) < f(\bar w) = g([\bar w_1;\; \bar w_2]))$ and
	\begin{equation}
		\norm{
			\begin{bmatrix}
				\tilde w_1 - \bar w_1 \\
				\tilde w_2 - \bar w_2
			\end{bmatrix}
		}_\infty
		=
		\norm{
			\begin{bmatrix}
				\tilde x_+ - \bar x_+ \\
				-(\tilde x_- - \bar x_-)
			\end{bmatrix}
		}_\infty
		\leq 
		\norm{\tilde x - \bar x}_\infty
	\end{equation}
	i.e.\ $(\tilde w_1, \tilde w_2)$ is arbitrarily close to $(\bar w_1, \bar w_2)$. This contradicts the assumption that $(\bar w_1, \bar w_2)$ is a local minimizer for \eqref{eqn:scotlass_reformulated}.

	Finally, since the optimal value of \eqref{eqn:scotlass_reformulated} is less than the one of \eqref{eqn:scotlass} and the global optimum $(\bar w_1, \bar w_2)$  of \eqref{eqn:scotlass_reformulated} defines a feasible $\bar x$ for \eqref{eqn:scotlass} with $ g([\bar w_1;\; \bar w_2]) = f(\bar x)$ we conclude that the optimal values of \eqref{eqn:scotlass} and \eqref{eqn:scotlass_reformulated} coincide.
\qed\end{proof}
Furthermore, the following complementarity condition holds which implies that the sparsity of a solution $\bar w$ of \eqref{eqn:scotlass_reformulated} is equal to the sparsity of the solution $\bar x = \bar w_1 - \bar w_2$ for Problem \eqref{eqn:scotlass}.
\begin{lemma} \label{lem:sparse_pca_ancillary}
	Every (local) minimizer $\bar w$ of \eqref{eqn:scotlass_reformulated} has $\bar w_1^T \bar w_2 = 0$.
\end{lemma}
\begin{proof}
	Assume the contrary, i.e. that $\bar w_1^T \bar w_2 \neq 0$. Due to feasibility of $\bar w = [\bar w_1;\; \bar w_2]$ we have $\bar w \geq 0$. Thus, there exists an $i \in \{1, \dots, n\}$ such that $\bar w_{(i)}, \bar w_{(i + n)} > 0$. Define $\tilde w$ element-wise as
	\begin{equation} \tilde w_{(k)} =
		\begin{cases}
		 \bar w_{(k)} - \epsilon, &k = i, i + n \\
		 \bar w_{(k)}, &\text{otherwise}
		\end{cases}
	\end{equation}
	for $\epsilon > 0$ sufficiently small so that $\tilde w > 0$. We will show that for a sufficiently small, positive $\alpha$ the point $(1 + \alpha) \tilde w$ is feasible and of smaller objective value than the local minimizer $\bar w$ while $\norm{\bar w - (1 + \alpha) \tilde w}$ can become arbitrarily small for an appropriate choice of $\alpha$ and $\epsilon$. This will be a contradiction of (local) optimality of $\bar w$.

	Indeed, note that $\mathbf{1}^T\tilde w < \mathbf{1}^T\bar w \leq \gamma$ and $\norm{\tilde w}_2 < \norm{\bar w}_2 \leq 1$, thus $(1 + \alpha) \tilde w$ is feasible for any $\alpha > 0$ sufficiently small. Moreover, $g(\tilde w) = g(\bar w)$ giving $g ((1+ \alpha) \tilde w) = (1 + \alpha)^2 g(\bar w) < g(\bar w)$, since $g(\bar w) < 0$ as otherwise $\bar w$ would be a global maximizer of $\eqref{eqn:scotlass_reformulated}$.
\qed\end{proof}

Finally, we present detailed comparison results of our algorithm against \texttt{Ipopt} on the problems described in Section \ref{sec:sqp}.

%% file: sections/tables_cutest.tex
\begin{sidewaystable} 
	\sisetup{round-mode=places}
	\caption{Comparison of the suggested active-set algorithm against \texttt{Ipopt} in calculating SQP search directions on 63 nonconvex problems of the \texttt{CUTEst} library. These include all the problems with linear inequality constraints, at most $n = 2000$ variables and an indefinite Hessian at the initial point. The initial point is computed with \texttt{GUROBI} as the nearest feasible point (in the two norm) to the initial point suggested by \texttt{CUTEst}. The column $m$ denotes the number of inequality constraints, while columns $t, \delta f^*, \text{error}$ denote the solution time (in seconds), the difference in objective value $f(\delta x^*) - f(0)$ achieved and the residual error defined by \eqref{eqn:error} of the suggested active-set solver. Similarly, $\delta f^*_{\text{ipopt}}, \text{error}_{\text{ipopt}}$ denote respectively identical metrics for \texttt{Ipopt}. $t_{\text{ipopt}}^{\text{sparse}}$ $(t_{\text{ipopt}}^{\text{dense}})$ denotes \texttt{Ipopt}'s timings when the matrices are passed in sparse (dense) format. Hardware used: Intel E5-2640v3, 64GB memory.}
	\label{tab:cutest}
	\centering
	\resizebox{\textwidth}{!}{
	\begin{tabular}{
		ll
	*{2}{S[]}
	*{3}{S[round-precision=2, table-format=1.2e-2]}
	*{2}{S[round-precision=2, table-format=-1.2e-2]}
	*{2}{S[round-precision=2, table-format=-1.2e-2]}
	}
	\toprule
	{Problem} & {Parameter} & {$n$}    & {$m$}    & {$t$}           & {$t_{\text{ipopt}}^{\text{dense}}$} & {$t_{\text{ipopt}}^{\text{sparse}}$}         & {$\delta f^*$}            & {$\delta f^*_{\text{ipopt}}$}           & {$\text{error}$}      & {$\text{error}_{\text{ipopt}}$}      \\
	\midrule
AVION2       &           & 49   & 113  & 4.444538E-02 & 4.670914E-02 & 1.723294E-02 & -3.645564E+02 & -3.645789E+02 & 1.479577E-10 & 5.426607E-04 \\
BLOCKQP1     & N=100     & 210  & 521  & 7.184425E-03 & 1.308548E+00 & 5.072931E-02 & -6.745688E+00 & -6.745688E+00 & 9.992007E-16 & 9.472623E-09 \\
BLOCKQP2     & N=100     & 210  & 521  & 1.333553E-02 & 1.234218E+00 & 5.287435E-02 & -6.745688E+00 & -6.745688E+00 & 6.661338E-16 & 9.374904E-09 \\
BLOCKQP3     & N=100     & 210  & 521  & 1.551006E-02 & 7.109685E-01 & 3.531169E-02 & -4.018526E+00 & -4.018526E+00 & 2.164935E-15 & 9.224197E-09 \\
BLOCKQP4     & N=100     & 210  & 521  & 9.190950E-03 & 7.812505E-01 & 4.003907E-02 & -4.018526E+00 & -4.018526E+00 & 1.015854E-14 & 8.997483E-09 \\
BLOCKQP5     & N=100     & 210  & 521  & 9.125178E-03 & 7.250970E-01 & 3.328339E-02 & -4.018526E+00 & -4.018526E+00 & 1.160183E-14 & 9.325761E-09 \\
BLOWEYA      & N=100     & 202  & 304  & 2.057082E-01 & 6.153741E-01 & 4.804298E-02 & -2.902540E-03 & -2.903560E-03 & 1.103075E-12 & 4.457334E-05 \\
BLOWEYB      & N=100     & 202  & 304  & 3.518280E-01 & 5.780202E-01 & 4.473963E-02 & -2.574329E-03 & -2.575152E-03 & 1.452606E-12 & 4.834189E-05 \\
BLOWEYC      & N=100     & 202  & 304  & 3.704172E-02 & 5.043310E-01 & 4.545440E-02 & -2.714783E-03 & -2.715942E-03 & 3.039952E-15 & 3.156699E-05 \\
EQC          &           & 9    & 21   & 2.374110E-04 & 3.362073E-02 & 3.607733E-02 & -1.653931E+00 & -4.903961E+03 & 0.000000E+00 & 5.729870E+00 \\
EXPFITA      &           & 5    & 22   & 9.989510E-04 & 5.855086E-03 & 7.309660E-03 & -1.529569E+01 & -1.529569E+01 & 7.813195E-12 & 1.159203E-07 \\
EXPFITB      &           & 5    & 102  & 8.978320E-04 & 1.018311E-02 & 1.349212E-02 & -7.615627E+01 & -7.615627E+01 & 1.547296E-11 & 2.935301E-06 \\
EXPFITC      &           & 5    & 502  & 2.641202E-03 & 4.953891E-02 & 6.193263E-02 & -3.810092E+02 & -3.810092E+02 & 8.583134E-13 & 1.278743E-04 \\
FERRISDC     & n=100     & 1100 & 2110 & 1.632799E+01 & 4.871856E+02 & 2.109886E+00 & -2.314228E-01 & -2.314213E-01 & 9.975523E-09 & 4.939870E-05 \\
GOULDQP1     &           & 32   & 81   & 4.174306E-03 & 2.652730E-02 & 1.418419E-02 & -8.790888E+02 & -8.790888E+02 & 9.783285E-13 & 7.851076E-05 \\
HIMMELBJ     &           & 45   & 61   & 4.420100E+00 & 3.329924E-02 & 2.115145E-02 & -7.390332E+01 & -7.391036E+01 & 3.586903E+00 & 1.144493E-05 \\
HS105        &           & 8    & 17   & 1.351519E-03 & 4.024936E-03 & 4.477112E-03 & -8.081024E+01 & -8.081024E+01 & 1.452172E-12 & 2.669379E-07 \\
HS24         &           & 2    & 5    & 3.546430E-04 & 6.277818E-03 & 6.520859E-03 & -1.311008E-01 & -1.311008E-01 & 2.220446E-16 & 1.623253E-08 \\
HS36         &           & 3    & 7    & 2.946660E-04 & 5.016933E-03 & 5.453507E-03 & -1.832051E+02 & -1.832051E+02 & 7.105427E-14 & 9.974060E-09 \\
HS37         &           & 3    & 8    & 2.950220E-04 & 5.691497E-03 & 4.989736E-03 & -1.832051E+02 & -1.832051E+02 & 9.947598E-14 & 9.974059E-09 \\
HS41         &           & 4    & 9    & 1.630200E-04 & 4.818620E-03 & 5.279881E-03 & -1.562500E-02 & -1.562500E-02 & 5.551115E-17 & 1.329137E-08 \\
HS44         &           & 4    & 10   & 4.904870E-04 & 8.343742E-03 & 6.473175E-03 & -1.304760E+00 & -1.304760E+00 & 2.442491E-15 & 3.577725E-08 \\
HS44NEW      &           & 4    & 10   & 6.557700E-04 & 4.320227E-03 & 4.416368E-03 & -1.847553E+00 & -1.847553E+00 & 2.220446E-15 & 8.037391E-09 \\
HS55         &           & 6    & 14   & 2.517480E-04 & 5.743211E-03 & 5.443675E-03 & -1.000000E+00 & -2.236068E+00 & 4.000000E+00 & 5.616085E-05 \\
	\end{tabular}}
\end{sidewaystable}
\begin{sidewaystable}
	\sisetup{round-mode=places}
	\caption{Continuation of Table \ref{tab:cutest}}
	\label{tab:cutest_2}
	\centering
	\resizebox{\textwidth}{!}{
	\begin{tabular}{
		ll
	*{2}{S[]}
	*{3}{S[round-precision=2, table-format=1.2e-2]}
	*{2}{S[round-precision=2, table-format=-1.2e-2]}
	*{2}{S[round-precision=2, table-format=-1.2e-2]}
	}
	\toprule
	{Problem} & {Parameter} & {$n$}    & {$m$}    & {$t$}           & {$t_{\text{ipopt}}^{\text{dense}}$} & {$t_{\text{ipopt}}^{\text{sparse}}$}         & {$\delta f^*$}            & {$\delta f^*_{\text{ipopt}}$}           & {$\text{error}$}      & {$\text{error}_{\text{ipopt}}$}      \\
	\midrule
LEUVEN7      &           & 360  & 1605 & 5.521594E-01 & 6.015317E+00 & 3.374051E-01 & -4.858191E+02 & -4.858192E+02 & 2.160050E-12 & 5.712026E-05 \\
LINCONT      &           & 1257 & 2435 & 6.700989E+01 & 4.709799E+02 & 6.443721E-01 & -1.113381E-18 & -1.558091E-05 & 2.899090E-10 & 3.035288E-07 \\
MPC10        &           & 1530 & 3571 & 2.982143E+01 & 5.949118E+02 & 9.300218E-01 & -3.155428E+03 & -3.155438E+03 & 1.636501E-11 & 9.900918E-05 \\
MPC11        &           & 1530 & 3571 & 3.138670E+01 & 6.484969E+02 & 9.605708E-01 & -3.155339E+03 & -3.155349E+03 & 2.258621E-11 & 1.353352E-04 \\
MPC12        &           & 1530 & 3571 & 3.015642E+01 & 6.625879E+02 & 9.935781E-01 & -3.155068E+03 & -3.155078E+03 & 1.451272E-11 & 9.204114E-05 \\
MPC13        &           & 1530 & 3571 & 3.139020E+01 & 6.049481E+02 & 8.909481E-01 & -3.154808E+03 & -3.154819E+03 & 1.516229E-11 & 1.890977E-04 \\
MPC14        &           & 1530 & 3571 & 3.036349E+01 & 6.181861E+02 & 9.324590E-01 & -3.154810E+03 & -3.154821E+03 & 2.793515E-11 & 1.047976E-04 \\
MPC15        &           & 1530 & 3571 & 3.230341E+01 & 5.987057E+02 & 9.142717E-01 & -3.153846E+03 & -3.153859E+03 & 2.210518E-11 & 1.069349E-04 \\
MPC16        &           & 1530 & 3571 & 3.241830E+01 & 6.266897E+02 & 9.186500E-01 & -3.153841E+03 & -3.153853E+03 & 1.616848E-11 & 1.070100E-04 \\
MPC2         &           & 1530 & 3571 & 2.824377E+01 & 5.600366E+02 & 8.767303E-01 & -3.157144E+03 & -3.157151E+03 & 1.161907E-11 & 1.056836E-04 \\
MPC3         &           & 1530 & 3571 & 2.783808E+01 & 6.103022E+02 & 9.163411E-01 & -3.157321E+03 & -3.157329E+03 & 1.224183E-11 & 1.047022E-04 \\
MPC4         &           & 1530 & 3571 & 3.018797E+01 & 6.347866E+02 & 9.388815E-01 & -3.156964E+03 & -3.156972E+03 & 1.548417E-11 & 1.033535E-04 \\
MPC5         &           & 1530 & 3571 & 2.943601E+01 & 6.344659E+02 & 9.331457E-01 & -3.157075E+03 & -3.157083E+03 & 1.574789E-11 & 1.044170E-04 \\
MPC6         &           & 1530 & 3571 & 3.035928E+01 & 5.720201E+02 & 9.081594E-01 & -3.151828E+03 & -3.151837E+03 & 1.558678E-11 & 1.905436E-04 \\
MPC7         &           & 1530 & 3571 & 2.897083E+01 & 6.455158E+02 & 9.593935E-01 & -3.155357E+03 & -3.155366E+03 & 1.857497E-11 & 1.770294E-04 \\
MPC8         &           & 1530 & 3571 & 2.761454E+01 & 5.973750E+02 & 9.020797E-01 & -3.155408E+03 & -3.155418E+03 & 1.359729E-11 & 1.503306E-04 \\
MPC9         &           & 1530 & 3571 & 2.949440E+01 & 6.456766E+02 & 9.470979E-01 & -3.155392E+03 & -3.155402E+03 & 1.521752E-11 & 1.044509E-04 \\
NASH         &           & 72   & 138  & 2.023284E-02 & 9.083129E-02 & 2.347952E-02 & -1.335068E-16 & -2.562575E-05 & 1.101341E-13 & 2.652532E-07 \\
NCVXQP1      & N=1000    & 1000 & 2500 & 3.468481E-01 & 5.037756E+01 & 2.530511E-01 & -7.824962E+04 & -7.824962E+04 & 9.094947E-12 & 3.340764E-07 \\
NCVXQP2      & N=1000    & 1000 & 2500 & 3.615064E-01 & 5.018531E+01 & 2.535988E-01 & -5.736680E+04 & -5.736680E+04 & 1.273293E-11 & 2.509192E-07 \\
NCVXQP3      & N=1000    & 1000 & 2500 & 3.681914E-01 & 4.789657E+01 & 2.396625E-01 & -4.998258E+04 & -4.998258E+04 & 1.273293E-11 & 2.732257E-07 \\
NCVXQP4      & N=1000    & 1000 & 2250 & 3.495667E-01 & 3.453693E+01 & 2.044327E-01 & -7.824962E+04 & -7.824962E+04 & 5.184120E-11 & 3.340765E-07 \\
NCVXQP5      & N=1000    & 1000 & 2250 & 3.553980E-01 & 3.450249E+01 & 2.057896E-01 & -5.736680E+04 & -5.736680E+04 & 1.714398E-09 & 2.509195E-07 \\
	\end{tabular}}
\end{sidewaystable}
\begin{sidewaystable}
	\sisetup{round-mode=places}
	\caption{Continuation of Table \ref{tab:cutest}}
	\label{tab:cutest_3}
	\centering
	\resizebox{\textwidth}{!}{
	\begin{tabular}{
		ll
	*{2}{S[]}
	*{3}{S[round-precision=2, table-format=1.2e-2]}
	*{2}{S[round-precision=2, table-format=-1.2e-2]}
	*{2}{S[round-precision=2, table-format=-1.2e-2]}
	}
	\toprule
	{Problem} & {Parameter} & {$n$}    & {$m$}    & {$t$}           & {$t_{\text{ipopt}}^{\text{dense}}$} & {$t_{\text{ipopt}}^{\text{sparse}}$}         & {$\delta f^*$}            & {$\delta f^*_{\text{ipopt}}$}           & {$\text{error}$}      & {$\text{error}_{\text{ipopt}}$}      \\
	\midrule
NCVXQP6      & N=1000    & 1000 & 2250 & 3.642895E-01 & 3.295662E+01 & 1.933023E-01 & -4.998258E+04 & -4.998258E+04 & 5.820766E-11 & 2.732277E-07 \\
NCVXQP7      & N=1000    & 1000 & 2750 & 3.584502E-01 & 7.042072E+01 & 2.941084E-01 & -7.824962E+04 & -7.824962E+04 & 5.129550E-10 & 3.340764E-07 \\
NCVXQP8      & N=1000    & 1000 & 2750 & 3.669888E-01 & 7.371150E+01 & 2.933839E-01 & -5.736680E+04 & -5.736680E+04 & 8.305506E-09 & 2.509189E-07 \\
NCVXQP9      & N=1000    & 1000 & 2750 & 3.835907E-01 & 6.705865E+01 & 2.787119E-01 & -4.998258E+04 & -4.998258E+04 & 2.119577E-08 & 2.732247E-07 \\
PENTAGON     &           & 6    & 15   & 1.087885E-03 & 9.631484E-03 & 9.260053E-03 & -2.699815E-02 & -2.266403E-02 & 8.267692E-15 & 2.092105E-06 \\
QC           &           & 9    & 22   & 2.120020E-04 & 1.348843E-02 & 1.216467E-02 & -1.114234E+02 & -1.114236E+02 & 0.000000E+00 & 1.875249E-05 \\
QCNEW        &           & 9    & 21   & 2.276460E-04 & 1.657338E-02 & 1.513761E-02 & -1.653931E+00 & -4.112503E+01 & 0.000000E+00 & 1.990100E-04 \\
QPNBLEND     &           & 83   & 157  & 1.501935E-02 & 6.937614E-02 & 2.064464E-02 & -9.282172E-03 & -9.283588E-03 & 2.504663E-13 & 4.679297E-05 \\
QPNBOEI1     &           & 384  & 891  & 6.116935E-01 & 1.115413E+01 & 1.653756E-01 & -4.366064E+01 & -4.366092E+01 & 7.786588E-12 & 8.287324E-05 \\
QPNBOEI2     &           & 143  & 363  & 4.073949E-02 & 1.011505E+00 & 6.225512E-02 & -1.003864E+03 & -1.004486E+03 & 1.224766E-09 & 8.420519E-05 \\
QPNSTAIR     &           & 467  & 905  & 7.413311E-01 & 1.500511E+01 & 2.324631E-01 & -1.150055E+02 & -1.150130E+02 & 4.352719E-11 & 2.310892E-04 \\
SOSQP1       & N=1000    & 2000 & 5001 & 1.025841E+00 & 3.387020E+02 & 2.187173E-01 & -5.000000E-01 & -5.000000E-01 & 3.108624E-14 & 4.572343E-09 \\
SOSQP2       & N=1000    & 2000 & 5001 & 1.189768E+00 & 2.171972E+03 & 1.061519E+00 & -5.000000E-01 & -5.000000E-01 & 4.940492E-15 & 4.988599E-09 \\
STATIC3      &           & 434  & 240  & 9.991469E-01 & 7.633038E-01 & 5.708679E-02 & -1.473959E+04 & -1.473959E+04 & 4.599613E-06 & 1.455835E-04 \\
STNQP1       & P=10      & 1025 & 2560 & 3.909290E-01 & 5.022181E+01 & 1.634204E-01 & -1.629919E+03 & -1.629919E+03 & 2.557954E-13 & 1.049186E-08 \\
STNQP2       & P=10      & 1025 & 2560 & 3.919327E-01 & 4.356643E+01 & 1.601439E-01 & -1.629919E+03 & -1.629919E+03 & 4.547474E-13 & 1.047931E-08
	\end{tabular}}
\end{sidewaystable}